\documentclass[11pt]{article}
\usepackage{amsmath,amsfonts,amssymb,amsthm,enumerate,graphicx}
\usepackage{tikz,authblk}
\usepackage{subfig}
\usepackage{url}

\usepackage{yfonts}

\newtheorem{theorem}{{\bf Theorem}}

\newtheorem{corollary}[theorem]{{\bf Corollary}}

\newtheorem{proposition}[theorem]{{\bf Proposition}}
\newtheorem{remark}[theorem]{{\bf Remark}}

\newcommand{\lk}[2]{{\rm lk}_{#1}(#2)}
\newcommand{\st}[2]{{\rm st}_{#1}(#2)}
\newcommand{\skel}[2]{{\rm skel}_{#1}(#2)}
\newcommand{\Kd}[1]{{\mathcal K}(#1)}
\newcommand{\bs}{\backslash}
\newcommand{\ol}{\overline}
\newcommand{\meets}{\leftrightarrow}
\newcommand{\nmeets}{\nleftrightarrow}

\newcommand{\FF}{ \ensuremath{\mathbb{F}}}
\newcommand{\QQ}{ \ensuremath{\mathbb{Q}}}
\newcommand{\ZZ}{ \ensuremath{\mathbb{Z}}}

\newcommand{\IntA}{A^{\!\!\!^{^{\circ}}}}
\newcommand{\IntB}{B^{\!\!\!^{^{\circ}}}}

\newcommand{\TPSS}{S^{\hspace{.2mm}2}\! \times \hspace{-3.3mm}_{-} \,
S^{\hspace{.1mm}1}}

\renewcommand{\Authfont}{\scshape}
\renewcommand{\Affilfont}{\itshape\small}
\renewcommand{\Authand}{,}
\renewcommand{\Authands}{, }

\begin{document}

\author{Basudeb Datta\thanks{Supported by SERB, DST (Grant No.\,MTR/2017\!/ 000410) and the UGC Centre for Advance Studies. 
}}

\affil{Department of Mathematics, Indian Institute of Science, Bangalore 560\,012, India.
 dattab@iisc.ac.in.}

\title{Vertex-transitive covers of semi-equivelar toroidal maps}

%\vspace{-15mm}

%\date{April 20, 2020}
\date{}

\maketitle

\vspace{-13mm}

\begin{abstract}
A map $X$ on a surface is called vertex-transitive if the automorphism group of $X$ acts transitively on the set of vertices of $X$.  If the face-cycles at all the vertices in a map  are of same type then the map is called semi-equivelar. In general, semi-equivelar maps on a surface form a bigger class than vertex-transitive maps. There are semi-equivelar toroidal maps which are not vertex-transitive. In this article, we  show that semi-equivelar toroidal maps are quotients of vertex-transitive toroidal maps. More explicitly, we prove that each semi-equivelar toroidal map  has a finite vertex-transitive cover.  In 2019, Drach {\em et al.} have shown that each vertex-transitive toroidal map has a minimal almost regular cover. Therefore, semi-equivelar toroidal maps are quotients of almost regular toroidal maps.
\end{abstract}

\noindent {\small {\em MSC 2010\,:} 52C20, 52B70, 51M20, 57M60.}


\section{Introduction}

Here all the maps are polyhedral maps on surfaces. Thus, a face of a map is an $n$-gon for some integer $n\geq 3$ and intersection of two intersecting faces is either a vertex or an edge.  Maps on the sphere and the torus are called {\em spherical} and {\em toroidal} maps respectively. For two maps $X$, $Y$, a surjection $\eta : Y \to X$ that preserves adjacency and sends vertices, edges and faces of $Y$ to vertices, edges and faces of $X$ respectively is called a {\em polyhedral covering} (or, in short, {\em covering}). Clearly, if $\eta : Y \to X$ is a covering then  $X$ is a quotient of $Y$. An {\em isomorphism} is an injective covering. We identify two maps if they are isomorphic.

The faces containing a vertex $u$ in a map $M$ form a cycle (the {\em face-cycle} at $u$) $C_u$ in the dual graph $\Lambda(M)$ of $M$. Clearly, $C_u$ is of the form $(F_{1,1}\mbox{-}\cdots \mbox{-}F_{1,n_1})\mbox{-}\cdots\mbox{-} (F_{k,1}\mbox{-}\cdots \mbox{-}F_{k,n_k})\mbox{-}F_{1,1}$, where $F_{i,j}$  is a  $p_i$-gon for $1\leq j \leq n_i$ and  $p_i\neq p_{i+1}$  for $1\leq i\leq k$  (addition in the suffix is modulo $k$).  
We say that the {\em vertex-type} of $u$ is  $[p_1^{n_1}, \dots, p_k^{n_k}]$. (We identify $[p_1^{n_1}, \dots, p_k^{n_k}]$ with $[p_k^{n_k}, \dots, p_1^{n_1}]$ and with $[p_2^{n_2}, \dots, p_k^{n_k}, p_1^{n_1}]$.)  A map $M$ is called {\em semi-equivelar} if  vertex-types of all the vertices are same.

A {\em semi-regular tiling} of a surface $S$ of constant curvature is a semi-equivelar map on $S$ in which each face  is a regular polygon and each edge is a geodesic. 
An {\em Archimedean tiling} of the plane $\mathbb{R}^2$ is a semi-regular tiling of the Euclidean plane. 
In \cite{GS1977}, Gr\"{u}nbaum and Shephard showed that there are exactly eleven types of Archimedean tilings of the plane. The vertex-types of these maps are $[3^6]$, $[4^4]$, $[6^3]$, $[3^4,6^1]$, $[3^3,4^2]$,  $[3^2,4^1,3^1,4^1]$, $[3^1,6^1,3^1,6^1]$, $[3^1,4^1,6^1,4^1]$, $[3^1,12^2]$,  $[4^1,6^1,12^1]$, $[4^1,8^2]$ respectively.  Clearly, the quotient of an Archimedean tiling of the plane by a discrete subgroup of the automorphism group of the tiling is a semi-equivelar map. This way we get eleven types of semi-equivelar toroidal maps. 
From \cite{DM2018}, we know the following converse.  

\begin{proposition} \label{prop:Archi}
 Each semi-equivelar toroidal map is $($isomorphic to$)$ a quotient of an Archimedean tiling of the plane. 
 \end{proposition}
 
 Thus each semi-equivelar toroidal map is isomorphic to a semi-regular tiling of the flat torus. 
A map $M$ is said to be {\em vertex-transitive} if the automorphism group ${\rm Aut}(M)$ acts transitively on the set $V(M)$ of vertices of $M$.  Clearly, vertex-transitive maps are semi-equivelar. 
There are infinitely many  vertex-transitive toroidal maps of each of the eleven vertex-types. 
From \cite{DM2017, DM2018, DU2005}, we know the following.  

\begin{proposition} \label{prop:vt} $(a)$ If the vertex-type of a semi-equivelar toroidal map $X$ is $[3^6]$, $[6^3]$,  $[4^4]$ or $[3^3,4^2]$ then $X$ is vertex-transitive. 
$(b)$ If $[p_1^{n_1}, \dots, p_k^{n_k}] = [3^2,4^1,3^1,4^1]$, $[3^4,6^1]$, $[3^1,6^1,3^1,6^1]$, $[3^1,4^1,6^1,4^1]$,  $[3^1,12^2]$, $[4^1,8^2]$  or $[4^1,6^1,12^1]$ then there exists a semi-equivelar toroidal map of vertex-type $[p_1^{n_1}, \dots, p_k^{n_k}]$  which is not vertex-transitive.
\end{proposition}

We know that boundary of the pseudo-rhombicuboctahedron is a semi-regular spherical map  of vertex-type $[4^3,3]$. Since the sphere $\mathbb{S}^2$ is simply connected, this map has  no vertex-transitive cover. 
%In \cite{M2020}, Maiti have shown that there exists a 30-vertex semi-equivelar map $M$  of vertex-type 
%$[5^3,3]$ on the double torus  which is not vertex-transitive and there does not exist any vertex-transitive %map of vertex-type $[5^3,3]$ on the (hyperbolic) plane. Thus, $M$ has no vertex-transitive universal %cover. 
Since the Archimedean tilings are vertex-transitive maps, each semi-equivelar toroidal map has vertex-transitive universal cover. Here we prove

\begin{theorem} \label{theo:se-vt}
If $X$ is a semi-equivelar toroidal map then there exists a covering $\eta : Y \to X$ where $Y$ is a vertex-transitive toroidal map.
\end{theorem} 

Thus, semi-equivelar toroidal maps are quotients of vertex-transitive toroidal maps. Since semi-equivelar maps on the Klein bottle are quotients of semi-equivelar toroidal maps, it follows from Theorem 
\ref{theo:se-vt} that semi-equivelar maps on the Klein bottle are quotients of vertex-transitive toroidal maps. 

A {\em flag} in a map $M$ is a triple $(v, e, f)$, where $f$ is a face of $M$, $e$ is an edge of $f$ and $v$ is a vertex of $e$. A map $X$ is called {\em regular} if  $\mbox{Aut}(X)$ acts transitively on the set of flags of $X$. Not all Archimedean tilings are regular. 
A semi-equivelar toroidal map $X$ is called {\em almost regular} if it has the same number of flag orbits under the action of its automorphism group as the Archimedean tiling $\widetilde{X}$ of the plane of same vertex-type has under the action of its symmetry group. From \cite{DHMS2019}, we know 

\begin{proposition} \label{prop:vt-ar} 
If $Y$ is a vertex-transitive toroidal map then there exists a covering $\xi : Z \to Y$ where $Z$ is an almost  regular toroidal map.
\end{proposition}

As a consequence of Theorem \ref{theo:se-vt} and Proposition \ref{prop:vt-ar}  we get 
\begin{corollary} \label{cor:se-ar}
If $X$ is a semi-equivelar toroidal map then there exists a covering $\mu : Z \to X$ where $Z$ is an almost regular toroidal map. 
\end{corollary}

\section{Proof of Theorem \ref{theo:se-vt}}

The seven diagrams in Figures $1, \dots, 4$ represent seven Archimedean tilings of the plane. These diagrams are  from \cite{DM2017, DM2018}. We have included these for the following proof. 

\begin{proof}[Proof of Theorem \ref{theo:se-vt}]
For $1\leq i\leq 7$, let $E_{i}$ be the Archimedean tiling given in Figures $1\dots, 4$. Let $V_{i} = V(E_i)$ be the vertex set of $E_{i}$. Let ${\mathcal H}_{i}$ be the group of all the translations of $E_{i}$. So, ${\mathcal H}_i \leq {\rm Aut}(E_{i})$.

Let $X$ be a semi-equivelar toroidal map. If the vertex-type of $X$ is $[3^6]$, $[6^3]$,  $[4^4]$ or $[3^3,4^2]$ then, by  Proposition \ref{prop:vt}, $X$ is vertex-transitive and there is nothing to prove. 

\smallskip 

\noindent {\sf Case 1.} Assume that the vertex-type of $X$ is $[4^1,8^2]$. By Proposition \ref{prop:Archi},  we can assume that $X = E_{1}/{\mathcal K}_{1}$ for some subgroup ${\mathcal K}_{1}$ of ${\rm Aut}(E_{1})$ and ${\mathcal K}_{1}$ has no fixed element (vertex, edge or face). Hence ${\mathcal K}_{1}$ consists of translations and glide reflections. Since $X = E_{1}/{\mathcal K}_{1}$ is orientable, ${\mathcal K}_{1}$ does not contain any glide reflection. Thus ${\mathcal K}_{1} \leq  {\mathcal H}_{1}$. Let $\eta_{1} : E_{1}\to X=E_1/{\mathcal K}_1$ be the canonical  covering.
We take origin $(0,0)$ to be the middle point of the line segment joining vertices $u_{0,0}$ and $u_{1,1}$ of $E_{1}$ (see Fig. 1 (a)). Let $A_1 := u_{1,0} - u_{0,0}$, $B_1 := u_{0,2} - u_{0,0} \in \mathbb{R}^2$. Then ${\mathcal H}_{1} =$ $ \langle \alpha_1 : z \mapsto z + A_1, \, \beta_1 : z \mapsto z + B_1 \rangle$. 

%%%%%%%%%%%%%%%%%%%E1&E2%%%%%%%%%%%%%%%%%%%%%%%

\begin{figure}[ht!]

\tiny
\tikzstyle{ver}=[]
\tikzstyle{vert}=[circle, draw, fill=black!100, inner sep=0pt, minimum width=4pt]
\tikzstyle{vertex}=[circle, draw, fill=black!00, inner sep=0pt, minimum width=4pt]
\tikzstyle{edge} = [draw,thick,-]
%\centering

%%%%%%%%%%%%%%%%%%%%E1%%%%%%%%%%%%%%%%%%%%%
\begin{tikzpicture}[scale=0.41]

\draw [yshift = -106] ({-3.7+2*cos(337.5)}, {1+2*sin(337.5)}) -- ({-3.7+2*cos(22.5)}, {2*sin(22.5)}) -- ({-3.7+2*cos(67.5)}, {2*sin(67.5)}) -- ({-5.7+2*cos(22.5)}, {2*sin(67.5)});

\draw [yshift = -106] ({2*cos(22.5)}, {2*sin(22.5)}) -- ({2*cos(67.5)}, {2*sin(67.5)}) -- ({2*cos(112.5)}, {2*sin(112.5)}) -- ({2*cos(157.5)}, {2*sin(157.5)});

\draw [yshift = -106] ({3.7+2*cos(337.5)}, {1+2*sin(337.5)}) -- ({3.7+2*cos(22.5)}, {2*sin(22.5)}) -- ({3.7+2*cos(67.5)}, {2*sin(67.5)}) -- ({3.7+2*cos(112.5)}, {2*sin(112.5)}) -- ({3.7+2*cos(157.5)}, {2*sin(157.5)}) -- ({3.7+2*cos(202.5)}, {1+2*sin(202.5)});

\draw [yshift = -106] ({7.4+2*cos(337.5)}, {1+ 2*sin(337.5)}) -- ({7.4+2*cos(22.5)}, {2*sin(22.5)}) -- ({7.4+2*cos(67.5)}, {2*sin(67.5)}) -- ({7.4+2*cos(112.5)}, {2*sin(112.5)}) -- ({7.4+2*cos(157.5)}, {2*sin(157.5)}) -- ({7.4+2*cos(202.5)}, {1+2*sin(202.5)});

\draw [yshift = -106] ({10.4+2*cos(67.5)}, {2*sin(67.5)}) -- ({11.1+2*cos(112.5)}, {2*sin(112.5)}) -- ({11.1+2*cos(157.5)}, {2*sin(157.5)});

%%%%%

\draw ({-3+2*cos(247.5)}, {2*sin(247.5)}) -- ({-3.7+2*cos(292.5)}, {2*sin(292.5)}) -- ({-3.7+2*cos(337.5)}, {2*sin(337.5)}) -- ({-3.7+2*cos(22.5)}, {2*sin(22.5)}) -- ({-3.7+2*cos(67.5)}, {2*sin(67.5)}) -- ({-5.7+2*cos(22.5)}, {2*sin(67.5)});

\draw ({2*cos(22.5)}, {2*sin(22.5)}) -- ({2*cos(67.5)}, {2*sin(67.5)}) -- ({2*cos(112.5)}, {2*sin(112.5)}) -- ({2*cos(157.5)}, {2*sin(157.5)}) -- ({2*cos(202.5)}, {2*sin(202.5)}) -- ({2*cos(247.5)}, {2*sin(247.5)}) -- ({2*cos(292.5)}, {2*sin(292.5)}) -- ({2*cos(337.5)}, {2*sin(337.5)}) -- ({2*cos(22.5)}, {2*sin(22.5)});

\draw ({3.7+2*cos(22.5)}, {2*sin(22.5)}) -- ({3.7+2*cos(67.5)}, {2*sin(67.5)}) -- ({3.7+2*cos(112.5)}, {2*sin(112.5)}) -- ({3.7+2*cos(157.5)}, {2*sin(157.5)}) -- ({3.7+2*cos(202.5)}, {2*sin(202.5)}) -- ({3.7+2*cos(247.5)}, {2*sin(247.5)}) -- ({3.7+2*cos(292.5)}, {2*sin(292.5)}) -- ({3.7+2*cos(337.5)}, {2*sin(337.5)}) -- ({3.7+2*cos(22.5)}, {2*sin(22.5)});

\draw ({7.4+2*cos(22.5)}, {2*sin(22.5)}) -- ({7.4+2*cos(67.5)}, {2*sin(67.5)}) -- ({7.4+2*cos(112.5)}, {2*sin(112.5)}) -- ({7.4+2*cos(157.5)}, {2*sin(157.5)}) -- ({7.4+2*cos(202.5)}, {2*sin(202.5)}) -- ({7.4+2*cos(247.5)}, {2*sin(247.5)}) -- ({7.4+2*cos(292.5)}, {2*sin(292.5)}) -- ({7.4+2*cos(337.5)}, {2*sin(337.5)}) -- ({7.4+2*cos(22.5)}, {2*sin(22.5)});

\draw ({10.4+2*cos(67.5)}, {2*sin(67.5)}) -- ({11.1+2*cos(112.5)}, {2*sin(112.5)}) -- ({11.1+2*cos(157.5)}, {2*sin(157.5)}) -- ({11.1+2*cos(202.5)}, {2*sin(202.5)}) -- ({11.1+2*cos(247.5)}, {2*sin(247.5)}) -- ({10.4+2*cos(292.5)}, {2*sin(292.5)});

%%%%

\draw [yshift = 106] ({-3+2*cos(247.5)}, {2*sin(247.5)}) -- ({-3.7+2*cos(292.5)}, {2*sin(292.5)}) -- ({-3.7+2*cos(337.5)}, {2*sin(337.5)}) -- ({-3.7+2*cos(22.5)}, {2*sin(22.5)}) -- ({-3.7+2*cos(67.5)}, {2*sin(67.5)}) -- ({-5.7+2*cos(22.5)}, {2*sin(67.5)});

\draw [yshift = 106] ({2*cos(22.5)}, {2*sin(22.5)}) -- ({2*cos(67.5)}, {2*sin(67.5)}) -- ({2*cos(112.5)}, {2*sin(112.5)}) -- ({2*cos(157.5)}, {2*sin(157.5)}) -- ({2*cos(202.5)}, {2*sin(202.5)}) -- ({2*cos(247.5)}, {2*sin(247.5)}) -- ({2*cos(292.5)}, {2*sin(292.5)}) -- ({2*cos(337.5)}, {2*sin(337.5)}) -- ({2*cos(22.5)}, {2*sin(22.5)});

\draw [yshift = 106] ({3.7+2*cos(22.5)}, {2*sin(22.5)}) -- ({3.7+2*cos(67.5)}, {2*sin(67.5)}) -- ({3.7+2*cos(112.5)}, {2*sin(112.5)}) -- ({3.7+2*cos(157.5)}, {2*sin(157.5)}) -- ({3.7+2*cos(202.5)}, {2*sin(202.5)}) -- ({3.7+2*cos(247.5)}, {2*sin(247.5)}) -- ({3.7+2*cos(292.5)}, {2*sin(292.5)}) -- ({3.7+2*cos(337.5)}, {2*sin(337.5)}) -- ({3.7+2*cos(22.5)}, {2*sin(22.5)});

\draw [yshift = 106] ({7.4+2*cos(22.5)}, {2*sin(22.5)}) -- ({7.4+2*cos(67.5)}, {2*sin(67.5)}) -- ({7.4+2*cos(112.5)}, {2*sin(112.5)}) -- ({7.4+2*cos(157.5)}, {2*sin(157.5)}) -- ({7.4+2*cos(202.5)}, {2*sin(202.5)}) -- ({7.4+2*cos(247.5)}, {2*sin(247.5)}) -- ({7.4+2*cos(292.5)}, {2*sin(292.5)}) -- ({7.4+2*cos(337.5)}, {2*sin(337.5)}) -- ({7.4+2*cos(22.5)}, {2*sin(22.5)});

\draw [yshift = 106] ({10.4+2*cos(67.5)}, {2*sin(67.5)}) -- ({11.1+2*cos(112.5)}, {2*sin(112.5)}) -- ({11.1+2*cos(157.5)}, {2*sin(157.5)}) -- ({11.1+2*cos(202.5)}, {2*sin(202.5)}) -- ({11.1+2*cos(247.5)}, {2*sin(247.5)}) -- ({10.4+2*cos(292.5)}, {2*sin(292.5)});
%\node[ver] () at (21,4){\tiny $u_{-1,-1}$};

%%%
\draw [yshift = 212] ({-3+2*cos(247.5)}, {2*sin(247.5)}) -- ({-3.7+2*cos(292.5)}, {2*sin(292.5)}) -- ({-3.7+2*cos(337.5)}, {2*sin(337.5)}) -- ({-3.7+2*cos(22.5)}, {2*sin(22.5)}) -- ({-3.7+2*cos(67.5)}, {2*sin(67.5)}) -- ({-5.7+2*cos(22.5)}, {2*sin(67.5)});

\draw [yshift = 212] ({2*cos(22.5)}, {2*sin(22.5)}) -- ({2*cos(67.5)}, {2*sin(67.5)}) -- ({2*cos(112.5)}, {2*sin(112.5)}) -- ({2*cos(157.5)}, {2*sin(157.5)}) -- ({2*cos(202.5)}, {2*sin(202.5)}) -- ({2*cos(247.5)}, {2*sin(247.5)}) -- ({2*cos(292.5)}, {2*sin(292.5)}) -- ({2*cos(337.5)}, {2*sin(337.5)}) -- ({2*cos(22.5)}, {2*sin(22.5)});

\draw [yshift = 212] ({3.7+2*cos(22.5)}, {2*sin(22.5)}) -- ({3.7+2*cos(67.5)}, {2*sin(67.5)}) -- ({3.7+2*cos(112.5)}, {2*sin(112.5)}) -- ({3.7+2*cos(157.5)}, {2*sin(157.5)}) -- ({3.7+2*cos(202.5)}, {2*sin(202.5)}) -- ({3.7+2*cos(247.5)}, {2*sin(247.5)}) -- ({3.7+2*cos(292.5)}, {2*sin(292.5)}) -- ({3.7+2*cos(337.5)}, {2*sin(337.5)}) -- ({3.7+2*cos(22.5)}, {2*sin(22.5)});

\draw [yshift = 212] ({7.4+2*cos(22.5)}, {2*sin(22.5)}) -- ({7.4+2*cos(67.5)}, {2*sin(67.5)}) -- ({7.4+2*cos(112.5)}, {2*sin(112.5)}) -- ({7.4+2*cos(157.5)}, {2*sin(157.5)}) -- ({7.4+2*cos(202.5)}, {2*sin(202.5)}) -- ({7.4+2*cos(247.5)}, {2*sin(247.5)}) -- ({7.4+2*cos(292.5)}, {2*sin(292.5)}) -- ({7.4+2*cos(337.5)}, {2*sin(337.5)}) -- ({7.4+2*cos(22.5)}, {2*sin(22.5)});

\draw [yshift = 212] ({10.4+2*cos(67.5)}, {2*sin(67.5)}) -- ({11.1+2*cos(112.5)}, {2*sin(112.5)}) -- ({11.1+2*cos(157.5)}, {2*sin(157.5)}) -- ({11.1+2*cos(202.5)}, {2*sin(202.5)}) -- ({11.1+2*cos(247.5)}, {2*sin(247.5)}) -- ({10.4+2*cos(292.5)}, {2*sin(292.5)});

\draw [yshift = 318] ({-3+2*cos(247.5)}, {2*sin(247.5)}) -- ({-3.7+2*cos(292.5)}, {2*sin(292.5)}) -- ({-3.7+2*cos(337.5)}, {2*sin(337.5)}) -- ({-3.7+2*cos(22.5)}, {-1+2*sin(22.5)});

\draw [yshift = 318] ({2*cos(202.5)}, {2*sin(202.5)}) -- ({2*cos(247.5)}, {2*sin(247.5)}) -- ({2*cos(292.5)}, {2*sin(292.5)}) -- ({2*cos(337.5)}, {2*sin(337.5)}) -- ({2*cos(22.5)}, {-1+2*sin(22.5)});

\draw [yshift = 318]  ({3.7+2*cos(202.5)}, {2*sin(202.5)}) -- ({3.7+2*cos(247.5)}, {2*sin(247.5)}) -- ({3.7+2*cos(292.5)}, {2*sin(292.5)}) -- ({3.7+2*cos(337.5)}, {2*sin(337.5)}) -- ({3.7+2*cos(22.5)}, {-1+2*sin(22.5)});

\draw [yshift = 318] ({7.4+2*cos(202.5)}, {2*sin(202.5)}) -- ({7.4+2*cos(247.5)}, {2*sin(247.5)}) -- ({7.4+2*cos(292.5)}, {2*sin(292.5)}) -- ({7.4+2*cos(337.5)}, {2*sin(337.5)}) -- ({7.4+2*cos(22.5)}, {-1+2*sin(22.5)});

\draw [yshift = 318] ({11.1+2*cos(202.5)}, {2*sin(202.5)}) -- ({11.1+2*cos(247.5)}, {2*sin(247.5)}) -- ({10.4+2*cos(292.5)}, {2*sin(292.5)});

\node[ver] () at (-1.8,-2){\tiny $v_{-3,-1}$};
\node[ver] () at (0,-1.5){\tiny $v_{-2,-1}$};
\node[ver] () at (1.8,-2){\tiny $v_{-1,-1}$};
\node[ver] () at (3.4,-1.5){\tiny $v_{0,-1}$};
\node[ver] () at (5.5,-2){\tiny $v_{1,-1}$};
\node[ver] () at (7.1,-1.5){\tiny $v_{2,-1}$};
\node[ver] () at (9.2,-2){\tiny $v_{3,-1}$};
\node[ver] () at (10.7,-1.5){\tiny $v_{4,-1}$};

\node[ver] () at (-1.8,-1.5+3.3){\tiny $v_{-3,0}$};
\node[ver] () at (0,-1.5+3.6){\tiny $v_{-2,0}$};
\node[ver] () at (1.8,-1.5+3.3){\tiny $v_{-1,0}$};
\node[ver] () at (3.4,-1.5+3.6){\tiny $v_{0,0}$};
\node[ver] () at (5.5,-1.5+3.3){\tiny $v_{1,0}$};
\node[ver] () at (7.1,-1.5+3.6){\tiny $v_{2,0}$};
\node[ver] () at (9.2,-1.5+3.3){\tiny $v_{3,0}$};
\node[ver] () at (10.7,-1.5+3.6){\tiny $v_{4,0}$};

\node[ver] () at (-1.8,5.5){\tiny $v_{-3,1}$};
\node[ver] () at (0,5.8){\tiny $v_{-2,1}$};
\node[ver] () at (1.8,5.5){\tiny $v_{-1,1}$};
\node[ver] () at (3.15,6){\tiny $v_{0,1}$};
\node[ver] () at (5.5,5.5){\tiny $v_{1,1}$};
\node[ver] () at (7.1,5.8){\tiny $v_{2,1}$};
\node[ver] () at (9.2,5.5){\tiny $v_{3,1}$};
\node[ver] () at (10.7,5.8){\tiny $v_{4,1}$};

\node[ver] () at (-1.8,9.3){\tiny $v_{-3,2}$};
\node[ver] () at (-.1,9.6){\tiny $v_{-2,2}$};
\node[ver] () at (1.8,9.3){\tiny $v_{-1,2}$};
\node[ver] () at (3.4,9.6){\tiny $v_{0,2}$};
\node[ver] () at (5.5,9.3){\tiny $v_{1,2}$};
\node[ver] () at (7.1,9.6){\tiny $v_{2,2}$};
\node[ver] () at (9.2,9.3){\tiny $v_{3,2}$};
\node[ver] () at (10.7,9.6){\tiny $v_{4,2}$};

%%%%%%%%%%%%%%%

\node[ver] () at (10,3+3.7+3.7){\tiny $u_{2,4}$};
\node[ver] () at (10,.8+3.6+3.7){\tiny $u_{2,3}$};
\node[ver] () at (10,3+3.7){\tiny $u_{2,2}$};
\node[ver] () at (10,.8+3.6){\tiny $u_{2,1}$};
\node[ver] () at (10,3){\tiny $u_{2,0}$};
\node[ver] () at (10.2,.8+3.6-3.7){\tiny $u_{2,-1}$};
\node[ver] () at (10.2,-3.7+3){\tiny $u_{2,-2}$};
\node[ver] () at (10.2,-3.2){\tiny $u_{2,-3}$};

\node[ver] () at (-1+7.2,3+3.7+3.7){\tiny $u_{1,4}$};
\node[ver] () at (-1+7.3,.8+3.6+3.7){\tiny $u_{1,3}$};
\node[ver] () at (-1+7.2,3+3.7){\tiny $u_{1,2}$};
\node[ver] () at (-2.5+7.3,.8+3.6){\tiny $u_{1,1}$};
\node[ver] () at (-1+7.2,3){\tiny $u_{1,0}$};
\node[ver] () at (-1+7.5,.8+3.5-3.7){\tiny $u_{1,-1}$};
\node[ver] () at (-1+7.4,-3.7+3){\tiny $u_{1,-2}$};
\node[ver] () at (6.6,-3.2){\tiny $u_{1,-3}$};

\node[ver] () at (-.9+3.6,3+3.7+3.7){\tiny $u_{0,4}$};
\node[ver] () at (-.9+3.6,.8+3.6+3.7){\tiny $u_{0,3}$};
\node[ver] () at (-2.5+3.6,3+3.7){\tiny $u_{0,2}$};
\node[ver] () at (-.9+3.6,.8+3.6){\tiny $u_{0,1}$};
\node[ver] () at (-.9+3.6,3){\tiny $u_{0,0}$};
\node[ver] () at (-.9+3.7,.8+3.6-3.7){\tiny $u_{0,-1}$};
\node[ver] () at (-.9+3.6,-3.7+3){\tiny $u_{0,-2}$};
\node[ver] () at (2.8,-3.2){\tiny $u_{0,-3}$};

\node[ver] () at (-.8,3+3.7+3.7){\tiny $u_{-2,4}$};
\node[ver] () at (-.8,.8+3.6+3.7){\tiny $u_{-2,3}$};
\node[ver] () at (-.8,3+3.7){\tiny $u_{-2,2}$};
\node[ver] () at (-.8,.8+3.6){\tiny $u_{-2,1}$};
\node[ver] () at (-.8,3){\tiny $u_{-2,0}$};
\node[ver] () at (-.7,.5){\tiny $u_{-2,-1}$};
\node[ver] () at (-.7,-.7){\tiny $u_{-2,-2}$};
\node[ver] () at (-.7,-3.2){\tiny $u_{-2,-3}$};

\draw [thick, dotted] (3.7,3.7) -- (7.4,3.7);
\draw [thick, dotted] (3.7,3.7) -- (3.7,7.4);

\node[ver] () at (3.7,3.7){$\bullet$};
\node[ver] () at (7.4,3.7){$\bullet$};
\node[ver] () at (3.7,7.4){$\bullet$}; 

\node[ver] () at (6.8,3.7){$\mathbf{>}$};
\node[ver] () at (6.8,4.2){ ${A_{1}}$}; 

\node[ver] () at (3.7,6.8){$\mathbf{\wedge}$};
\node[ver] () at (4.3,6.8){ ${B_{1}}$};

\node[ver] () at (3.5, -4.5){\normalsize (a): Truncated square tiling ${E}_{1}$};

%\end{scope}
\end{tikzpicture}\hfill
%%%%%%%%%%%%%%%E2%%%%%%%%%%%%%%%%%%%%%%%%
\begin{tikzpicture}[scale=0.2]
\begin{scope}[shift={(-15,-7)}]

\draw[edge, thin](5,5)--(10,5)--(10,10)--(5,10)--(5,5);
\draw[edge, thin](15,2.5)--(20,2.5)--(20,7.5)--(15,7.5)--(15,2.5);
\draw[edge, thin](25,0)--(30,0)--(30,5)--(25,5)--(25,0);
\draw[edge, thin](10,5)--(15,2.5);
\draw[edge, thin](10,10)--(15,7.5);
\draw[edge, thin](10,5)--(15,7.5);
\draw[edge](4,5.5)--(5,5)--(10,5)--(15,2.5)--(20,2.5)--(25,0)--(30,0)--(31,-0.5);

\draw[edge](12,-8.5)--(12.5,-7.5)--(12.5,-2.5)--(15,2.5)--(15,7.5)--(17.5,12.5)--(17.5,17.5)--(18,18.5);

\node[ver] () at (10.3,-5.2){\scriptsize $u_{-1,-2}$};
\node[ver] () at (15,-7){\scriptsize $u_{0,-2}$};
\node[ver] () at (20.2,-7.7){\scriptsize $u_{1,-2}$};

\node[ver] () at (10.2,0){\scriptsize $u_{-1,-1}$};
\node[ver] () at (15.1,-3.5){\scriptsize $u_{0,-1}$};
\node[ver] () at (19.8,-2.7){\scriptsize $u_{1,-1}$};
\node[ver] () at (24.9,-4.5){\scriptsize $u_{2,-1}$};
\node[ver] () at (29.7,-5){\scriptsize $u_{3,-1}$};

\node[ver] () at (7,5.5){\scriptsize $u_{-2,0}$};
\node[ver] () at (13,4.8){\scriptsize $u_{-1,0}$};
\node[ver] () at (13,1.8){\scriptsize $u_{0,0}$};
\node[ver] () at (22.5,2.3){\scriptsize $u_{1,0}$};
\node[ver] () at (26.7,.5){\scriptsize $u_{2,0}$};
\node[ver] () at (32,-.2){\scriptsize $u_{3,0}$};

\node[ver] () at (7.5,10.5){\scriptsize $u_{-2,1}$};
\node[ver] () at (12.5,10){\scriptsize $u_{-1,1}$};
\node[ver] () at (17,8){\scriptsize $u_{0,1}$};
\node[ver] () at (22,7.3){\scriptsize $u_{1,1}$};
\node[ver] () at (27,5.5){\scriptsize $u_{2,1}$};
\node[ver] () at (32,4.8){\scriptsize $u_{3,1}$};

\node[ver] () at (10,14.2){\scriptsize $u_{-2,2}$};
\node[ver] () at (15.5,14.6){\scriptsize $u_{-1,2}$};
\node[ver] () at (19.6,11.7){\scriptsize $u_{0,2}$};
\node[ver] () at (25,12.2){\scriptsize $u_{1,2}$};
\node[ver] () at (29.1,10.6){\scriptsize $u_{2,2}$};

\node[ver] () at (19.2,16.7){\scriptsize $u_{0,3}$};
\node[ver] () at (24.7,17.3){\scriptsize $u_{1,3}$};
\node[ver] () at (29.1,14){\scriptsize $u_{2,3}$};

 \draw[edge, thin](20,2.5)--(25,0);
 \draw[edge, thin](20,7.5)--(25,5);
 \draw[edge, thin](20,2.5)--(25,5);
 \draw[edge, thin](30,5)--(31,7);
 \draw[edge, thin](30,5)--(31,4.5);
 \draw[edge, thin](30,0)--(31,0.5);
%\draw[edge, thin](30,0)--(31,-1);
%\draw[edge, thin](5,5)--(4,6);
 \draw[edge, thin](5,10)--(4,10.4);

 \draw[edge, thin](10,10)--(7.5,15)--(5,10);
 \draw[edge, thin](10,10)--(7.5,15)--(12.5,15)--(10,10);
 \draw[edge, thin](20,7.5)--(17.5,12.5)--(15,7.5);
 \draw[edge, thin](20,7.5)--(17.5,12.5)--(22.5,12.5)--(20,7.5);
 \draw[edge, thin](30,5)--(27.5,10)--(25,5);
 \draw[edge, thin](12.5,15)--(17.5,12.5);
 \draw[edge, thin](22.5,12.5)--(27.5,10);
 \draw[edge, thin](17.5,17.5)--(12.5,15);
\draw[edge, thin](17.5,12.5)--(22.5,12.5)--(22.5,17.5)--(17.5,17.5)--(17.5,12.5);
 \draw[edge, thin](22.5,17.5)--(27.5,15)--(22.5,12.5);
 \draw[edge, thin](27.5,15)--(27.5,10);
 \draw[edge, thin](27.5,15)--(31,15);
 \draw[edge, thin](27.5,10)--(31,10);
 \draw[edge, thin](7.5,15)--(5.5,16);
 \draw[edge, thin](12.5,15)--(12.5,18);
 \draw[edge, thin](17.5,17.5)--(15.5,18.5);
 \draw[edge, thin](5,10)--(4,9.5);
 \draw[edge, thin](5,5)--(3.5,3);
 \draw[edge, thin](7.5,15)--(7.5,18);
%\draw[edge, thin](17.5,17.5)--(18.5,19);
 \draw[edge, thin](22.5,17.5)--(23.5,19);
 \draw[edge, thin](22.5,17.5)--(21.5,19);
 \draw[edge, thin](27.5,15)--(28.5,17);

 \draw[edge, thin](10,5)--(7.5,0)--(5,5);
 \draw[edge, thin](10,5)--(7.5,0)--(12.5,-2.5)--(15,2.5);
 \draw[edge, thin](20,2.5)--(17.5,-2.5)--(15,2.5);
 \draw[edge, thin](17.5,-2.5)--(12.5,-2.5);
 \draw[edge, thin](30,0)--(27.5,-5)--(25,0);
 \draw[edge, thin](27.5,-5)--(22.5,-5)--(25,0);
 \draw[edge, thin](22.5,-5)--(17.5,-2.5);
\draw[edge, thin](17.5,-2.5)--(12.5,-2.5)--(12.5,-7.5)--(17.5,-7.5)--(17.5,-2.5);
 \draw[edge, thin](22.5,-5)--(17.5,-7.5);
 \draw[edge, thin](12.5,-2.5)--(12.5,-7.5)--(7.5,-5)--(12.5,-2.5);
 \draw[edge, thin](7.5,-5)--(7.5,0);
 \draw[edge, thin](7.5,-5)--(6.5,-7);
 \draw[edge, thin](7.5,-5)--(4.5,-5);
 \draw[edge, thin](7.5,0)--(4.5,0);
 \draw[edge, thin](22.5,-5)--(22.5,-7.5);
 \draw[edge, thin](17.5,-7.5)--(19.5,-8.5);
 \draw[edge, thin](17.5,-7.5)--(17,-8.5);
 \draw[edge, thin](12.5,-7.5)--(13,-8.5);
%\draw[edge, thin](12.5,-7.5)--(11.5,-8.5);
%\draw[edge, thin](27.5,-5)--(28.5,-6.5);
 \draw[edge, thin](27.5,-5)--(27.5,-7);
 \draw[edge, thin](27.5,-5)--(29.5,-6);

 \node[ver] () at (17.5,5){\scriptsize $\bullet$};
\node[ver] () at (27.5,2.5){\scriptsize $\bullet$};
\node[ver] () at (7.5,7.5){\scriptsize $\bullet$};

\node[ver] () at (15,-5){\scriptsize $\bullet$};
\node[ver] () at (20,15){\scriptsize $\bullet$};
%\draw[edge, thin](16.5,5)--(18.5,5);
%\draw[edge, thin](17.5,4)--(17.5,6);
 \draw [dashed] (3.5,8.5) -- (31.5,1.5);
 %\draw [dashed] (15,-5) -- (20,15);
\draw [dashed] (14,-9) -- (21,19); 

\node[ver] () at (26.5,2.8){$\mathbf{>}$}; 
\node[ver] () at (19.8,14){$\mathbf{\wedge}$}; 

\node[ver] () at (15,-12){\normalsize (b): Snub square tiling $E_2$};

%\node[ver] () at (15,-14){\mbox{}};

\end{scope}
\end{tikzpicture}
\vspace{-5mm}
\caption{}
\end{figure}
%%%%%%%%%%%%%%%%%%%%E1&E2%%%%%%%%%%%%%%%%%%%%%%%%

Let $\rho_1$ be the function obtained by  90 degrees anticlockwise rotation. Then $\rho_1\in {\rm Aut}(E_{1})$ and $\rho_1(A_1) = B_1$ and $\rho_1(B_1) = -A_1$. 
Consider the group ($\leq {\rm Aut}(E_{1})$) 
\begin{align*}
  {\mathcal G}_{1} & := \langle \alpha_1, \beta_1, \rho_1 \rangle.
\end{align*}

Clearly,  vertices of $E_{1}$ form four ${\mathcal H}_{1}$-orbits. These are  $\mbox{Orb}(u_{0,0}) = \mbox{Orb}(u_{1,0})$,   $\mbox{Orb}(v_{1,0}) = \mbox{Orb}(v_{1,1})$,  $\mbox{Orb}(u_{0,1}) = \mbox{Orb}(u_{1,1}) $ and $\mbox{Orb}(v_{0,0}) = \mbox{Orb}(v_{0,1})$. 
Since $\rho_1(u_{1,0})= v_{1,1}$, $\rho_1(v_{1,1})= u_{0,1}$, $\rho_1(u_{0,1})= v_{0,0}$, it follows that the vertices of $E_1$ form one ${\mathcal G}_1$-orbit. So, ${\mathcal G}_1$ acts vertex-transitively on $E_1$.

Since ${\mathcal K}_1\leq {\mathcal H}_1$,  it follows that ${\mathcal K}_1 = \langle \gamma_1 : z\mapsto z+C_1, \, \delta_1 : z\mapsto z+D_1 \rangle$, where $C_1=aA_1+bB_1$ and $D_1 =cA_1 + dB_1$, for some $a, b, c, d\in \mathbb{Z}$. 

\medskip

\noindent {\bf Claim 1.} There exists a factor $m$ of $|ad-bc|$ such that ${\mathcal L}_1:= \langle \alpha_1^m, \alpha_2^m\rangle \leq {\mathcal K}_1$. 

\medskip

Since $E_1/{\mathcal K}_1$ is compact, $C_1$ and $D_1$ are linearly independent. Therefore, there exists $r, s, t, u\in \mathbb{Q}$ such that $A_1=rC_1+sD_1$ and $B_1 = tC_1 +uD_1$.  Suppose, $r=r_1/r_2$, $s=s_1/s_2$, $t=t_1/t_2$, $u=u_1/u_2$, where $r_1$, $r_2$, $s_1$, $s_2$, $t_1$, $t_2$, $u_1$, $u_2$ are integers and $r_2, s_2, t_2, u_2>0$. (Also, if $r\neq 0$ then $r_1$ and $r_2$ are relatively prime. Similarly, for $s, t, u$.) Let $m := \mbox{lcm}\{r_2, s_2, t_2, u_2\}$.  Then $mr, ms, mt, mu\in \mathbb{Z}$  and $mA_1 = (mr)C_1+(ms)D_1$, $mB_1 = (mt)C_1 +(mu)D_1$. Therefore, $\alpha_1^m(z) = z+mA_1 = (\gamma_1^{mr}\circ\delta_1^{ms})(z)$, 
$\beta_1^m(z) = z+mB_1 = (\gamma_1^{mt}\circ\delta_1^{mu})(z)$ and hence $\alpha_1^m, \beta_1^m\in {\mathcal K}_1$.   Observe that $(r,s,t,u) = \frac{1}{ad-bc}(d, -b, -c, a)$. So, $r_2, s_2, t_2, u_2$ are factors of $|ad-bc|$ and hence $m$ is a factor of $|ad-bc|$. (Observe that $m$ is the smallest positive integer so that  $mr, ms, mt, mu$ are integers.) This proves Claim 1. 

\medskip

Since ${\mathcal H}_1$ is abelian, we have ${\mathcal L}_1\unlhd {\mathcal K}_1\unlhd {\mathcal H}_1\leq {\mathcal G}_1\leq {\rm Aut}(E_1)$. 

\medskip

\noindent {\bf Claim 2.}  ${\mathcal L}_1 \unlhd {\mathcal G}_{1}$.

\smallskip

For $x, y\in \mathbb{R}$ and $p\in \mathbb{Z}$,  $(\rho_1\circ\alpha_1^p\circ\rho_1^{-1})(xA_1+ yB_1) = (\rho_1\circ\alpha_1^p)(yA_1 - xB_1) = \rho_1((yA_1 - xB_1) + pA_1)
= \rho_1((p+  y)A_1 - xB_1)= (p + y)B_1 + xA_1 = (xA_1 + yB_1) + pB_1 = \alpha_2^p(xA_1 + yB_1)$. Thus, $\rho_1\circ\alpha_1^p\circ\rho_1^{-1} = \alpha_2^p$. 
Again, for $x, y\in \mathbb{R}$, $(\rho_1\circ\alpha_2^p\circ\rho_1^{-1})(xA_1 + yB_1) = (\rho_1\circ\alpha_2^p)(yA_1 - xB_1) = \rho_1((yA_1- xB_1) +pB_1) 
= \rho_1(yA_1 + (p - x)B_1) = yB_1 - (p - x)A_1 = (xA_1 + yB_1) -pA_1 = \alpha_1^{-p}(xA_1 + yB_1)$. Thus, $\rho_1\circ\alpha_2^p\circ\rho_1^{-1} = \alpha_1^{-p}$. 
In particular, $\rho_1\circ\alpha_1^m\circ\rho_1^{-1} = \alpha_2^m$, $\rho_1\circ\alpha_2^m\circ\rho_1^{-1} = \alpha_1^{-m} \in {\mathcal L}_1$. Since $\alpha_1$ and $\alpha_2$ commute, these prove  Claim 2.
(If we take $p=1$, we get ${\mathcal H}_1 \unlhd {\mathcal G}_1$.) 

\smallskip

By Claim 2, ${\mathcal G}_1/{\mathcal L}_1$ is a group. 
Since ${\mathcal L}_1 \unlhd {\mathcal K}_1$,  $X=E_1/{\mathcal K}_1$ is a toroidal map and ${\mathcal L}_1$ is generated by two independent vectors, it follows that $Y := E_1/{\mathcal L}_1$ is a toroidal map and $v+ {\mathcal L}_1 \mapsto v + {\mathcal K}_1$ gives a (natural) covering $\eta: Y \to X$. Since the action of ${\mathcal G}_1$ on $E_1$ is vertex-transitive, it follows that  the action of ${\mathcal G}_1/{\mathcal L}_1$ on $Y = E_1/{\mathcal L}_1$ is also vertex-transitive. This proves  the result when the vertex-type of $X$ is $[4^1,8^2]$. 

\medskip 

\noindent {\sf Case 2.}  Now, assume that the vertex-type of $X$ is $[3^2,4^1,3^1,4^1]$. We take $(0,0)$ as the middle point of $u_{0,0}$ and $u_{1,1}$ of $E_2$ and  $A_2 = u_{2,0} - u_{0,0}$, $B_2 = u_{0,2}- u_{0,0}$  (see Fig. 1 (b)). The result follows in this case by exactly same arguments as Case 1. 

\medskip 

\noindent {\sf Case 3.}  Next assume that the vertex-type of $X$ is $[3^4, 6^1]$.  By Proposition \ref{prop:Archi}, we can assume that $X = E_{3}/{\mathcal K}_{3}$ for some subgroup ${\mathcal K}_{3}$ of ${\rm Aut}(E_{3})$. By the similar arguments as above, ${\mathcal K}_{3}$ is a subgroup of ${\mathcal H}_{3}$ and  $\eta_{3} : E_{3}\to X$ is a covering.

We take $(0,0)$ as the middle point of the line segment joining vertices $u_{-1,0}$ and $u_{1,0}$ of $E_{3}$ (see Fig. 2). Let $A_3 := u_{2,-1} - u_{-1,0}$, $B_3 := u_{0,2} - u_{-1,0}$, $F_3 := u_{-3,3} - u_{-1,0} \in \mathbb{R}^2$. Then $A_3 + F_3 = B_3$ and hence ${\mathcal H}_{3} = \langle \alpha_3 : z \mapsto z + A_3, \, \beta_3 : z \mapsto z + B_3 \rangle$.

%%%%%%%%%%%%%%%%%%%%%E3%%%%%%%%%%%%%%%%%%%%

\setlength{\unitlength}{2.3mm}
%\begin{picture}(58,24)(-24,-10.5)
\begin{picture}(58,26)(-6,-8)

\thinlines

\put(-6,-3){\line(1,0){17}}\put(19,-3){\line(1,0){15}}
\put(-6,0){\line(1,0){7}}\put(9,0){\line(1,0){20}}
\put(-1,3){\line(1,0){20}}\put(27,3){\line(1,0){7}}
\put(-6,6){\line(1,0){15}}\put(17,6){\line(1,0){17}}
\put(-6,9){\line(1,0){5}}\put(7,9){\line(1,0){20}}
\put(-6,12){\line(1,0){23}}\put(25,12){\line(1,0){9}}
\put(-6,15){\line(1,0){13}}\put(15,15){\line(1,0){19}}

\put(10.7,14.85){\mbox{${}_\bullet$}} %(11,15)
\put(12.7,5.85){\mbox{${}_\bullet$}} %(13,6)
\put(20.7,11.85){\mbox{${}_\bullet$}} %(21,12)
\put(22.7,2.85){\mbox{${}_\bullet$}} %(23,3)

\put(11.7,14.85){\mbox{\small $F_3$}} %(11,15)
\put(21.7,11.85){\mbox{\small $B_3$}} %(21,12)
\put(23.7,2.85){\mbox{\small $A_3$}} %(23,3)

%\put(10.6,14.7){\mbox{$\bullet$}} %(11,15)
%\put(12.6,5.7){\mbox{$\bullet$}} %(13,6)
%\put(20.6,11.7){\mbox{$\bullet$}} %(21,12)
%\put(22.6,2.7){\mbox{$\bullet$}} %(23,3)

%\put(16.77,5.62){\mbox{$\cdot$}} %(17,6)
\put(13.27,5.45){\mbox{$\cdot$}} %(13.5,5.85)
\put(13.77,5.3){\mbox{$\cdot$}} %(14,5.7)
\put(14.27,5.15){\mbox{$\cdot$}} %(14.5,5.55)
\put(14.77,5){\mbox{$\cdot$}} %(15,5.4)
\put(15.27,4.85){\mbox{$\cdot$}} %(15.5,5.25)
\put(15.77,4.7){\mbox{$\cdot$}} %(16,5.1)
\put(16.27,4.55){\mbox{$\cdot$}} %(16.5,4.95)
\put(16.77,4.4){\mbox{$\cdot$}} %(17,4.8)
\put(17.27,4.25){\mbox{$\cdot$}} %(17.5,4.65)
\put(17.77,4.1){\mbox{$\cdot$}} %(18,4.5)
\put(18.27,3.95){\mbox{$\cdot$}} %(18.5,4.35)
\put(18.77,3.8){\mbox{$\cdot$}} %(19,4.2)
\put(19.27,3.65){\mbox{$\cdot$}} %(19.5,4.05)
\put(19.77,3.5){\mbox{$\cdot$}} %(20,3.9)
\put(20.27,3.35){\mbox{$\cdot$}} %(20.5,3.75)
\put(20.77,3.2){\mbox{$\cdot$}} %(21,3.6)
\put(21.27,3.05){\mbox{$\cdot$}} %(21.5,3.45)
\put(21.77,2.9){\mbox{$\cdot$}} %(22,3.3)
\put(22.27,2.75){\mbox{$\cdot$}} %(22.5,3.15)

\put(13.17,5.9){\mbox{$\cdot$}} %(13.4,6.3)
\put(13.57,6.2){\mbox{$\cdot$}} %(13.8,6.6)
\put(13.97,6.5){\mbox{$\cdot$}} %(14.2,6.9)
\put(14.37,6.8){\mbox{$\cdot$}} %(14.6,7.2)
\put(14.77,7.1){\mbox{$\cdot$}} %(15,7.5)
\put(15.17,7.4){\mbox{$\cdot$}} %(15.4,7.8)
\put(15.57,7.7){\mbox{$\cdot$}} %(15.8,8.1)
\put(15.97,8){\mbox{$\cdot$}} %(16.2,8.4)
\put(16.37,8.3){\mbox{$\cdot$}} %(16.6,8.7)
\put(16.77,8.6){\mbox{$\cdot$}} %(17,9)
\put(17.17,8.9){\mbox{$\cdot$}} %(17.4,9.3)
\put(17.57,9.2){\mbox{$\cdot$}} %(17.8,9.6)
\put(17.97,9.5){\mbox{$\cdot$}} %(18.2,9.9)
\put(18.37,9.8){\mbox{$\cdot$}} %(18.6,10.2)
\put(18.77,10.1){\mbox{$\cdot$}} %(19,10.5)
\put(19.17,10.4){\mbox{$\cdot$}} %(19.4,10.8)
\put(19.57,10.7){\mbox{$\cdot$}} %(19.8,11.1)
\put(19.97,11){\mbox{$\cdot$}} %(20.2,11.4)
\put(20.37,11.3){\mbox{$\cdot$}} %(20.6,11.7)

\put(10.87,14.15){\mbox{$\cdot$}} %(11.1,14.55)
\put(10.97,13.7){\mbox{$\cdot$}} %(11.2,14.1)
\put(11.07,13.25){\mbox{$\cdot$}} %(11.3,13.65)
\put(11.17,12.8){\mbox{$\cdot$}} %(11.4,13.2)
\put(11.27,12.35){\mbox{$\cdot$}} %(11.5,12.75)
\put(11.37,11.9){\mbox{$\cdot$}} %(11.6,12.2)
\put(11.47,11.45){\mbox{$\cdot$}} %(11.7,11.85)
\put(11.57,11){\mbox{$\cdot$}} %(11.8,11.4)
\put(11.67,10.55){\mbox{$\cdot$}} %(11.9,10.95)
\put(11.77,10.1){\mbox{$\cdot$}} %(12,10.5)

\put(11.87,9.65){\mbox{$\cdot$}} %(12.1,10.05)
\put(11.97,9.2){\mbox{$\cdot$}} %(12.2,9.6)
\put(12.07,8.75){\mbox{$\cdot$}} %(12.3,9.15)
\put(12.17,8.3){\mbox{$\cdot$}} %(12.4,8.7)
\put(12.27,7.85){\mbox{$\cdot$}} %(12.5,8.25)
\put(12.37,7.4){\mbox{$\cdot$}} %(12.6,7.8)
\put(12.47,6.95){\mbox{$\cdot$}} %(12.7,7.35)
\put(12.57,5.5){\mbox{$\cdot$}} %(12.8,6.9)
\put(12.67,6.05){\mbox{$\cdot$}} %(12.9,6.45)

%\put(10.65,14.6){$\bullet$}
%\put(20.65,11.61){$\bullet$}\put(8.65,11.61){$\circ$}\put(16.65,11.61){$\circ$}
%\put(30.65,8.6){$\bullet$}\put(2.65,8.6){$\bullet$}
%\put(12.65,5.6){$\bullet$} \put(4.65,5.6){$\circ$} \put(20.65,5.6){$\circ$}
%\put(22.65,2.59){$\bullet$} \put(-5.35,2.59){$\bullet$}
%\put(4.65,-.4){$\bullet$}\put(8.65,-.4){$\circ$} \put(16.65,-.4){$\circ$}\put(32.65,-.4){$\bullet$}
%\put(14.64,-3.4){$\bullet$}

\put(-5.65,14){\line(2,3){1.3}}
\put(-5.65,8){\line(2,3){5.3}}
\put(-3,6){\line(2,3){6.6}}
\put(-5.65,-4){\line(2,3){6.65}}\put(5,12){\line(2,3){2.5}}
\put(-1.65,-4){\line(2,3){10.65}}
\put(2,-4.5){\line(2,3){1}}\put(7,3){\line(2,3){8.7}}
\put(6.35,-4){\line(2,3){4.7}} \put(15,9){\line(2,3){4.7}}
\put(10.35,-4){\line(2,3){8.65}} \put(23,15){\line(2,3){1}}
\put(17,0){\line(2,3){10.7}}
\put(18.35,-4){\line(2,3){2.7}} \put(25,6){\line(2,3){6.5}}
\put(22.35,-4){\line(2,3){6.7}} \put(33,12){\line(2,3){1}}
\put(26.35,-4){\line(2,3){7.3}}
\put(30,-4.5){\line(2,3){1}}

\put(-4.35,-4){\line(-2,3){1.5}}
\put(-.35,-4){\line(-2,3){2.7}}
\put(3.65,-4){\line(-2,3){9.5}}
\put(8,-4.5){\line(-2,3){1}} \put(3,3){\line(-2,3){8.5}}
\put(11.65,-4){\line(-2,3){6.7}}\put(1,12){\line(-2,3){2.7}}
\put(13,0){\line(-2,3){10.7}}
\put(19.65,-4){\line(-2,3){4.7}}\put(11,9){\line(-2,3){4.7}}
\put(23.65,-4){\line(-2,3){10.7}}
\put(27.65,-4){\line(-2,3){2.7}}\put(21,6){\line(-2,3){6.7}}
\put(31.65,-4){\line(-2,3){8.7}} \put(19,15){\line(-2,3){1}}
\put(31,3){\line(-2,3){8.7}}
\put(29,12){\line(-2,3){2.7}} \put(34,4.5){\line(-2,3){1}}
\put(34,10.5){\line(-2,3){3.6}}

 \put(-4.5,-2.6){\mbox{\tiny $u_{-3,-3}$}}
 \put(-.5,-2.6){\mbox{\tiny $u_{-2,-3}$}}
 \put(3.5,-2.6){\mbox{\tiny $u_{-1,-3}$}}
 \put(7.5,-2.6){\mbox{\tiny $u_{0,-3}$}}
 \put(11.5,-2.6){\mbox{\tiny $u_{1,-3}$}}
 %\put(15.7,-2.6){\mbox{\tiny $u_{2,-3}$}}
 \put(19.5,-2.6){\mbox{\tiny $u_{3,-3}$}}
 \put(23.5,-2.6){\mbox{\tiny $u_{4,-3}$}}
 \put(27.5,-2.6){\mbox{\tiny $u_{5,-3}$}}
 \put(31,-2.6){\mbox{\tiny $u_{6,-3}$}}

 \put(-2.5,.4){\mbox{\tiny $u_{-3,-2}$}}
 \put(1.5,.4){\mbox{\tiny $u_{-2,-2}$}}
 %\put(5.5,.4){\mbox{\tiny $u_{-1,-2}$}}
 \put(9.5,.4){\mbox{\tiny $u_{0,-2}$}}
 \put(13.5,.4){\mbox{\tiny $u_{1,-2}$}}
 \put(17.5,.4){\mbox{\tiny $u_{2,-2}$}}
 \put(21.5,.4){\mbox{\tiny $u_{3,-2}$}}
 \put(25.5,.4){\mbox{\tiny $u_{4,-2}$}}
 \put(29.5,.4){\mbox{\tiny $u_{5,-2}$}}
 %\put(33.5,.4){\mbox{\tiny $u_{6,-2}$}}

%\put(8,1.5){\line(-2,3){6}}
%\put(-8.5,3.4){\mbox{\tiny $u_{-5,-1}$}}
%\put(-4.5,3.4){\mbox{\tiny $u_{-4,-1}$}}
 \put(-.5,3.4){\mbox{\tiny $u_{-3,-1}$}}
 \put(3.5,3.4){\mbox{\tiny $u_{-2,-1}$}}
 \put(7.5,3.4){\mbox{\tiny $u_{-1,-1}$}}
 \put(11,3.4){\mbox{\tiny $u_{0,-1}$}}
 \put(15.5,3.4){\mbox{\tiny $u_{1,-1}$}}
 \put(19.5,2.5){\mbox{\tiny $u_{2,-1}$}}
 %\put(23.5,3.4){\mbox{\tiny $u_{3,-1}$}}
 \put(27.5,3.4){\mbox{\tiny $u_{4,-1}$}}
 \put(31.5,3.4){\mbox{\tiny $u_{5,-1}$}}

 \put(-2.5,6.4){\mbox{\tiny $u_{-4,0}$}}
 \put(1 ,6.4){\mbox{\tiny $u_{-3,0}$}}
 \put(5.6,6.4){\mbox{\tiny $u_{-2,0}$}}
 \put(9.6,6.4){\mbox{\tiny $u_{-1,0}$}}
 %\put(13.4,6.4){\mbox{\tiny $u_{0,0}$}}
 \put(17.4,6.4){\mbox{\tiny $u_{1,0}$}}
 \put(21.4,6.4){\mbox{\tiny $u_{2,0}$}}
 \put(25.4,6.4){\mbox{\tiny $u_{3,0}$}}
 \put(29 ,6.4){\mbox{\tiny $u_{4,0}$}}
 \put(33.4,6.4){\mbox{\tiny $u_{5,0}$}}

 \put(-4.2,9.4){\mbox{\tiny $u_{-5,1}$}}
 \put(-.3,9.4){\mbox{\tiny $u_{-4,1}$}}
 %\put(3.7,9.4){\mbox{\tiny $u_{-3,1}$}}
 \put(7.7,9.4){\mbox{\tiny $u_{-2,1}$}}
 \put(11.7,9.4){\mbox{\tiny $u_{-1,1}$}}
 \put(15.6,9.4){\mbox{\tiny $u_{0,1}$}}
 \put(19,9.4){\mbox{\tiny $u_{1,1}$}}
 \put(23.6,9.4){\mbox{\tiny $u_{2,1}$}}
 \put(27.6,9.4){\mbox{\tiny $u_{3,1}$}}
 %\put(31.6,9.4){\mbox{\tiny $u_{4,1}$}}

 \put(-2.4,12.4){\mbox{\tiny $u_{-5,2}$}}
 \put(1.6,12.4){\mbox{\tiny $u_{-4,2}$}}
 \put(5.6,12.4){\mbox{\tiny $u_{-3,2}$}}
 \put(9.3,12.4){\mbox{\tiny $u_{-2,2}$}}
 \put(13.6,12.4){\mbox{\tiny $u_{-1,2}$}}
 \put(17.6,12.4){\mbox{\tiny $u_{0,2}$}}
 %\put(21.6,12.4){\mbox{\tiny $u_{1,2}$}}
 \put(25.6,12.4){\mbox{\tiny $u_{2,2}$}}
 \put(29.6,12.4){\mbox{\tiny $u_{3,2}$}}
 \put(33.6,12.4){\mbox{\tiny $u_{4,2}$}}

 \put(-4.4,15.4){\mbox{\tiny $u_{-6,3}$}}
 \put(-.4,15.4){\mbox{\tiny $u_{-5,3}$}}
 \put(3.6,15.4){\mbox{\tiny $u_{-4,3}$}}
 \put(7.6,15.4){\mbox{\tiny $u_{-3,3}$}}
 %\put(11.6,15.4){\mbox{\tiny $u_{-2,3}$}}
 \put(15.6,15.4){\mbox{\tiny $u_{-1,3}$}}
 \put(19.6,15.4){\mbox{\tiny $u_{0,3}$}}
 \put(23.6,15.4){\mbox{\tiny $u_{1,3}$}}
 \put(27.6,15.4){\mbox{\tiny $u_{2,3}$}}
 \put(31.6,15.4){\mbox{\tiny $u_{3,3}$}}
 %\put(35.6,15.4){\mbox{\tiny $u_{4,3}$}}

%\dashline[500]{10}[1](13,6)(11.2,15)
%\draw [dashed] (17.5,8.8) -- (23.8,12);
%\draw [dashed] (17.5,8.8) -- (24.8,6.8);
%\put(12.5,6){\mbox{\tiny $\bullet$}}
%\node[ver] () at (16.2,13.7){$\bullet$};
%\node[ver] () at (23.8,12){$\bullet$};
%\node[ver] () at (24.8,6.8){$\bullet$};

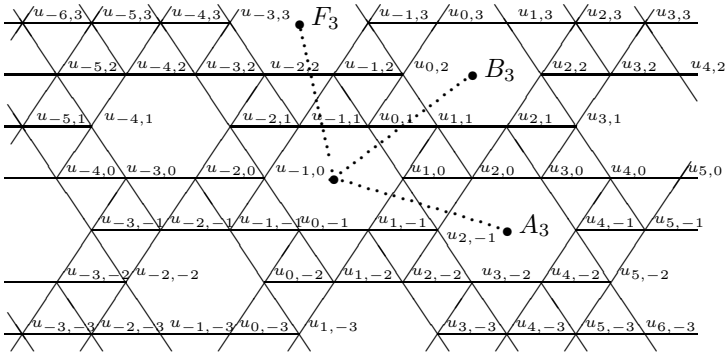
\put(6,-6.5) {Figure 2:  Snub hexagonal tiling $E_3$}
%\caption{}
\end{picture}
%%%%%%%%%%%%%%%%%%%%%%%%E3%%%%%%%%%%%%%%%%%%%%%%%%

Let $\rho_3$ be the function  obtained by  60 degrees anticlockwise rotation. Then $\rho_3 \in {\rm Aut}(E_{3})$ and  
$\rho_3(A_3) = B_3$, $\rho_3(B_3) = F_3$ and $\rho_3(F_3) = - A_3$. 
Consider the  group  ($\leq {\rm Aut}(E_{3})$)
\begin{align*}
  {\mathcal G}_{3} & := \langle \alpha_3, \beta_3, \rho_3 \rangle.
\end{align*}

Clearly,  vertices of $E_{3}$ form six ${\mathcal H}_{3}$-orbits. These are  $\mbox{Orb}(u_{1,0})$,   $\mbox{Orb}(u_{0,1})$, $\mbox{Orb}(u_{-1,1})$,    $\mbox{Orb}(u_{-1,0})$ , $\mbox{Orb}(u_{0,-1})$ and   $\mbox{Orb}(u_{1,-1})$. 
Since $\rho_3(u_{1,0})= u_{0,1}$, $\rho_3(u_{0,1})= u_{-1,1}$, $\rho_3(u_{-1,1})= u_{-1,0}$, $\rho_3(u_{-1,0}) = u_{0,-1}$, $\rho_3(u_{0,-1})= u_{1,-1}$  (see Fig. 2), it follows that vertices of $E_3$ form one ${\mathcal G}_3$-orbit. So, ${\mathcal G}_3$ acts vertex-transitively on $E_3$.

Since ${\mathcal K}_3\leq {\mathcal H}_3$,  ${\mathcal K}_3 = \langle \gamma_3 : z \mapsto z + C_3, \, \delta_3 : z \mapsto z + D_3 \rangle$, where $C_3 = aA_3 + bB_3$ and $D_3 = cA_3 + dB_3$, for some $a, b, c, d\in \mathbb{Z}$. 

\medskip

\noindent {\bf Claim 3.} There exists a factor $m$ of $|ad-bc|$ such that ${\mathcal L}_3 := \langle \alpha_3^m, \beta_3^m \rangle \leq {\mathcal K}_3$. 

\medskip

Since $E_3/{\mathcal K}_3$ is compact, $C_3$ and $D_3$ are linearly independent. Therefore,  there exists $r, s, t, u\in \mathbb{Q}$ such that $A_3 = rC_3 + sD_3$ and $B_3 = tC_3 + uD_3$.  As in the proof of Claim 1, let $m$ be the smallest positive integer such that  $mr, ms, mt, mu\in \mathbb{Z}$. Then $m$ is a factor of $|ad-bc|$ and $mA_3 = (mr)C_3 + (ms)D_3$, $mB_3 = (mt)C_3 + (mu)D_3$. Thus, $\alpha_3^m(z) = z + mA_3 = (\gamma_3^{mr}\circ\delta_3^{ms})(z)$, 
$\beta_3^m(z) = z+mB_3 = (\gamma_3^{mt}\circ\delta_3^{mu})(z)$ and hence $\alpha_3^m, \beta_3^m\in {\mathcal K}_3$.  This proves Claim 3. 

\medskip

Again, we have ${\mathcal L}_3\unlhd {\mathcal K}_3\unlhd {\mathcal H}_3\leq {\mathcal G}_3\leq {\rm Aut}(E_3)$. 

\medskip

\noindent {\bf Claim 4.}  ${\mathcal L}_3 \unlhd {\mathcal G}_{3}$.

\smallskip

For $x, y\in \mathbb{R}$,  $(\rho_3\circ\alpha_3^m\circ \rho_3^{-1})(xA_3+ yB_3) = ( \rho_3\circ\alpha_3^m)(x(-F_3)+yA_3) = 
( \rho_3\circ\alpha_3^m)(x(A_3-B_3)+yA_3) = ( \rho_3\circ\alpha_3^m)((x+y)A_3 - xB_3) = 
 \rho_3((x+y)A_3 - xB_3+ mA_3) =  \rho_3((m+x+y)A_3 - xB_3) =  (m+x+y)B_3 -xF_3 = (m+x+y)B_3 - x(B_3- A_3) = (xA_3 + yB_3) + mB_3 = \beta_3^m(xA_3 + yB_3)$. Thus, $\rho_3\circ\alpha_3^m\circ \rho_3^{-1} = \beta_3^m\in {\mathcal L}_3$. 
 Again,  for $x, y\in \mathbb{R}$,  $(\rho_3\circ\beta_3^m\circ \rho_3^{-1})(xA_3+ yB_3) = ( \rho_3\circ\beta_3^m)(x(- F_3) + yA_3) = ( \rho_3\circ\beta_3^m)(x(A_3 - B_3) + yA_3) = ( \rho_3\circ\beta_3^m)((x + y)A_3 - xB_3) = \rho_3((x + y)A_3 - xB_3 + mB_3) =  \rho_3((x+y)A_3 +(m- x)B_3) =  (x+y)B_3 +(m-x)F_3 = (x+y)B_3 +(m- x)(B_3- A_3) = (xA_3 + yB_3) - mA_3 = \alpha_3^{-m}(xA_3 + yB_3)$. Thus, $\ \rho_3\circ\beta_3^m\circ \rho_3^{-1} = \alpha_3^{-m} \in {\mathcal L}_3$. 
Since $\alpha_1$ and $\alpha_2$ commute, these prove  Claim 4.

\smallskip

By Claim 4, ${\mathcal G}_3/{\mathcal L}_3$ is a group. 
Since ${\mathcal L}_3 \unlhd {\mathcal K}_3$, $X = E_3/{\mathcal K}_3$ is a toroidal map 
and ${\mathcal L}_3$ is generated by two independent vectors, it follows that $Y:= E_3/{\mathcal L}_3$ is a toroidal map and $v+ {\mathcal L}_3 \mapsto v + {\mathcal K}_3$ gives a (natural) covering $\eta: Y \to X$. Since the action of ${\mathcal G}_3$ on $E_3$ is vertex-transitive, it follows that  the action of ${\mathcal G}_3/{\mathcal L}_3$ on $Y = E_3/{\mathcal L}_3$ is also vertex-transitive. This proves the result when the vertex-type of $X$ is $[3^4,6^1]$. 

%%%%%%%%%%%%%%%E4&E5%%%%%%%%%%%%%%

\begin{figure}[ht!]

\tiny
\tikzstyle{ver}=[]
\tikzstyle{vert}=[circle, draw, fill=black!100, inner sep=0pt, minimum width=4pt]
\tikzstyle{vertex}=[circle, draw, fill=black!00, inner sep=0pt, minimum width=4pt]
\tikzstyle{edge} = [draw,thick,-]
\centering

%%%%%%%%%%%%%%%E4%%%%%%%%%%%%%%%%%%
\begin{tikzpicture}[scale=.14]

\draw[edge, thin](-1.5,0)--(51.5,0);
\draw[edge, thin](-1.5,10)--(51.5,10);
\draw[edge, thin](-1.5,20)--(51.5,20);
\draw[edge, thin](-1.5,30)--(51.5,30);

\draw[edge, thin](-0.5,19)--(5.5,31);

\draw[edge, thin](-.5,-1)--(15.5,31);
\draw[edge, thin](9.5,-1)--(25.5,31);
\draw[edge, thin](19.5,-1)--(35.5,31);
\draw[edge, thin](29.5,-1)--(45.5,31);
\draw[edge, thin](39.5,-1)--(51.5,23);
\draw[edge, thin](49.5,-1)--(51.5,3);
%\draw[edge, thin](59.5,-1)--(70.5,21);

%\draw[edge, thin](40.5,5)--(45,0)

\draw[edge, thin](5.5,-1)--(-0.5,11);
\draw[edge, thin](15.5,-1)--(-0.5,31);
\draw[edge, thin](25.5,-1)--(9.5,31);
\draw[edge, thin](35.5,-1)--(19.5,31);
\draw[edge, thin](45.5,-1)--(29.5,31);

\draw[edge, thin](50.5,9)--(39.5,31);

\draw[edge, thin](51.5,27)--(49.5,31);

\draw [dashed] (22.5,15) -- (32.5,15);
\draw [dashed] (22.5,15) -- (17.5,25);
\draw [dashed] (22.5,15) -- (27.5,25);

\node[ver] () at (22.5,15){\scriptsize $\bullet$};
\node[ver] () at (32.5,15){\scriptsize $\bullet$};
\node[ver] () at (27.5,25){\scriptsize $\bullet$};
\node[ver] () at (17.5,25){\scriptsize $\bullet$};

\node[ver] () at (2.5,-1.7){\scriptsize ${w_{-2,-1}}$};
\node[ver] () at (7.5,1.2){\scriptsize ${v_{-1,-1}}$};
\node[ver] () at (12.5,-1.7){\scriptsize ${w_{-1,-1}}$};
\node[ver] () at (17.5,1.2){\scriptsize ${v_{0,-1}}$};
\node[ver] () at (22.5,-1.7){\scriptsize ${w_{0,-1}}$};
\node[ver] () at (27.5,1.2){\scriptsize ${v_{1,-1}}$};
\node[ver] () at (32.5,-1.7){\scriptsize ${w_{1,-1}}$};
\node[ver] () at (37.5,1.2){\scriptsize ${v_{2,-1}}$};
\node[ver] () at (42.5,-1.7){\scriptsize ${w_{2,-1}}$};
\node[ver] () at (47.5,1.2){\scriptsize ${v_{3,-1}}$};
\node[ver] () at (52.5,-1.7){\scriptsize ${w_{3,-1}}$};
%\node[ver] () at (57.5,1.2){\scriptsize ${v_{4,-1}}$};
%\node[ver] () at (62.5,-1.7){\scriptsize ${w_{5,-1}}$};
%\node[ver] () at (67.5,1.2){\scriptsize ${v_{5,-1}}$};

\node[ver] () at (6.3,5){\scriptsize ${u_{-2,-1}}$};
\node[ver] () at (16.3,5){\scriptsize ${u_{-1,-1}}$};
\node[ver] () at (25.5,5){\scriptsize ${u_{0,-1}}$};
\node[ver] () at (35.5,5){\scriptsize ${u_{1,-1}}$};
\node[ver] () at (45.5,5){\scriptsize ${u_{2,-1}}$};
%\node[ver] () at (55.5,5){\scriptsize ${u_{3,-1}}$};
%\node[ver] () at (65.5,5){\scriptsize ${u_{4,-1}}$};

\node[ver] () at (2.9,11){\scriptsize ${v_{-2,0}}$};
\node[ver] () at (7.5,8.7){\scriptsize ${w_{-2,0}}$};
\node[ver] () at (12.5,11){\scriptsize ${v_{-1,0}}$};
\node[ver] () at (17.5,8.7){\scriptsize ${w_{-1,0}}$};
\node[ver] () at (22,11){\scriptsize ${v_{0,0}}$};
\node[ver] () at (27,8.7){\scriptsize ${w_{0,0}}$};
\node[ver] () at (32,11){\scriptsize ${v_{1,0}}$};
\node[ver] () at (37,8.7){\scriptsize ${w_{1,0}}$};
\node[ver] () at (42,11){\scriptsize ${v_{2,0}}$};
\node[ver] () at (47,8.7){\scriptsize ${w_{2,0}}$};
\node[ver] () at (52,11){\scriptsize ${v_{3,0}}$};
%\node[ver] () at (57,9){\scriptsize ${w_{3,0}}$};
%\node[ver] () at (62,11){\scriptsize ${v_{4,0}}$};
%\node[ver] () at (67,9){\scriptsize ${w_{4,0}}$};

\node[ver] () at (5,15){\scriptsize ${u_{-2,0}}$};
\node[ver] () at (15,15){\scriptsize ${u_{-1,0}}$};
\node[ver] () at (30,13.5){\scriptsize ${u_{0,0}}$};
\node[ver] () at (40.5,15){\scriptsize ${u_{1,0}}$};
\node[ver] () at (50,15){\scriptsize ${u_{2,0}}$};
%\node[ver] () at (55,15){\scriptsize ${u_{3,0}}$};
%\node[ver] () at (65,15){\scriptsize ${u_{4,0}}$};

\node[ver] () at (2.5,18.5){\scriptsize ${w_{-1,1}}$};
\node[ver] () at (8,20.7){\scriptsize ${v_{-2,1}}$};
\node[ver] () at (13,18.5){\scriptsize ${w_{-2,1}}$};
\node[ver] () at (18,20.7){\scriptsize ${v_{-1,1}}$};
\node[ver] () at (23,18.5){\scriptsize ${w_{-1,1}}$};
\node[ver] () at (27.5,20.7){\scriptsize ${v_{0,1}}$};
\node[ver] () at (32,18.5){\scriptsize ${w_{0,1}}$};
\node[ver] () at (37.5,20.7){\scriptsize ${v_{1,1}}$};
\node[ver] () at (42,18.5){\scriptsize ${w_{1,1}}$};
\node[ver] () at (47.5,20.7){\scriptsize ${v_{2,1}}$};
\node[ver] () at (52,18.5){\scriptsize ${w_{2,1}}$};
%\node[ver] () at (57.5,20.7){\scriptsize ${v_{3,1}}$};
%\node[ver] () at (62.5,20.7){\scriptsize ${w_{3,1}}$};
%\node[ver] () at (67.5,20.7){\scriptsize ${v_{4,1}}$};

\node[ver] () at (6,25){\scriptsize ${u_{-3,1}}$};
\node[ver] () at (15.5,24){\scriptsize ${u_{-2,1}}$};
\node[ver] () at (25.5,24){\scriptsize ${u_{-1,1}}$};
\node[ver] () at (35.5,25){\scriptsize ${u_{0,1}}$};
\node[ver] () at (45,25){\scriptsize ${u_{1,1}}$};
%\node[ver] () at (55,25){\scriptsize ${u_{2,1}}$};
%\node[ver] () at (65,25){\scriptsize ${u_{3,1}}$};

\node[ver] () at (2.5,31){\scriptsize ${v_{-3,2}}$};
\node[ver] () at (7.7,28.5){\scriptsize ${w_{-3,2}}$};
\node[ver] () at (12.5,31){\scriptsize ${v_{-2,2}}$};
\node[ver] () at (17.5,28.5){\scriptsize ${w_{-2,2}}$};
\node[ver] () at (22.5,31){\scriptsize ${v_{-1,2}}$};
\node[ver] () at (27.5,28.5){\scriptsize ${w_{-1,2}}$};
\node[ver] () at (32.5,31){\scriptsize ${v_{0,2}}$};
\node[ver] () at (37.5,28.5){\scriptsize ${w_{0,2}}$};
\node[ver] () at (42.5,31){\scriptsize ${v_{1,2}}$};
\node[ver] () at (47.5,28.5){\scriptsize ${w_{1,2}}$};
\node[ver] () at (52.5,31){\scriptsize ${v_{2,2}}$};
%\node[ver] () at (57.5,29){\scriptsize ${w_{2,2}}$};
%\node[ver] () at (62.5,31.7){\scriptsize ${v_{3,2}}$};
%\node[ver] () at (67.5,29){\scriptsize ${w_{3,2}}$};

\node[ver] () at (28,-6) {\small (a): Trihexagonal tiling $E_4$};

\end{tikzpicture}
%%%%%%%%%%%%%%%%%%%%%%%E5%%%%%%%%%%%%%%%%%%
\begin{tikzpicture}[scale=0.35]

\draw ({sqrt(3)}, 1) -- (0, 2) -- ({-sqrt(3)}, 1) -- ({-sqrt(3)}, -1) -- (0, -2) -- ({sqrt(3)}, -1) -- ({sqrt(3)}, 1);
\draw ({6+sqrt(3)}, 1) -- (6+0, 2) -- ({6-sqrt(3)}, 1) -- ({6-sqrt(3)}, -1) -- (6+0, -2) -- ({6+sqrt(3)}, -1) -- ({6+sqrt(3)}, 1);
\draw ({12+sqrt(3)}, 1) -- (12+0, 2) -- ({12-sqrt(3)}, 1) -- ({12-sqrt(3)}, -1) -- (12+0, -2) -- ({12+sqrt(3)}, -1) -- ({12+sqrt(3)}, 1);

\draw ({-4.8+sqrt(3)}, 1) -- ({-sqrt(3)}, 1);
\draw ({-4.8+sqrt(3)}, -1) -- ({-sqrt(3)}, -1);
\draw ({sqrt(3)}, 1) -- ({6-sqrt(3)}, 1);
\draw ({sqrt(3)}, -1) -- ({6-sqrt(3)}, -1);
\draw ({6+sqrt(3)}, 1) -- ({12-sqrt(3)}, 1);
\draw ({6+sqrt(3)}, -1) -- ({12-sqrt(3)}, -1);
\draw ({12+sqrt(3)}, 1) -- ({16.8-sqrt(3)}, 1);
\draw ({12+sqrt(3)}, -1) -- ({16.8-sqrt(3)}, -1);

\draw ({-sqrt(3)}, 1) -- (-3+0, 2.6);
\draw (0, 2) -- (-1.25, 3.6);
\draw (0, 2) -- (1.3, 3.6);
\draw ({sqrt(3)}, 1) -- (3.05, 2.6);

\draw ({6-sqrt(3)}, 1) -- (3+0, 2.6);
\draw (6, 2) -- (4.75, 3.6);
\draw (6, 2) -- (7.3, 3.6);
\draw ({6+sqrt(3)}, 1) -- (9.05, 2.6);

\draw ({12-sqrt(3)}, 1) -- (9, 2.6);
\draw (12, 2) -- (10.75, 3.6);
\draw (12, 2) -- (13.3, 3.6);
\draw ({12+sqrt(3)}, 1) -- (15.05, 2.6);

%%%%

\draw [xshift = 86, yshift = 130] (-6+0, -2) -- ({-6+sqrt(3)}, -1) -- ({-6+sqrt(3)}, 1) -- ({-6+sqrt(3)}, 1) -- (-6+0, 2);
\draw [xshift = 86, yshift = 130] ({sqrt(3)}, 1) -- (0, 2) -- ({-sqrt(3)}, 1) -- ({-sqrt(3)}, -1) -- (0, -2) -- ({sqrt(3)}, -1) -- ({sqrt(3)}, 1);
\draw [xshift = 86, yshift = 130] ({6+sqrt(3)}, 1) -- (6+0, 2) -- ({6-sqrt(3)}, 1) -- ({6-sqrt(3)}, -1) -- (6+0, -2) -- ({6+sqrt(3)}, -1) -- ({6+sqrt(3)}, 1);
\draw [xshift = 86, yshift = 130] (12+0, 2) -- ({12-sqrt(3)}, 1) -- ({12-sqrt(3)}, -1) -- (12+0, -2) ;

\draw [xshift = 86, yshift = 130] ({-6+sqrt(3)}, 1) -- ({-sqrt(3)}, 1);
\draw [xshift = 86, yshift = 130] ({-6+sqrt(3)}, -1) -- ({-sqrt(3)}, -1);
\draw [xshift = 86, yshift = 130] ({sqrt(3)}, 1) -- ({6-sqrt(3)}, 1);
\draw [xshift = 86, yshift = 130] ({sqrt(3)}, -1) -- ({6-sqrt(3)}, -1);
\draw [xshift = 86, yshift = 130] ({6+sqrt(3)}, 1) -- ({12-sqrt(3)}, 1);
\draw [xshift = 86, yshift = 130] ({6+sqrt(3)}, -1) -- ({12-sqrt(3)}, -1);

%%%\tiny

\draw [yshift = 260] ({sqrt(3)}, 1) -- (0, 2) -- ({-sqrt(3)}, 1) -- ({-sqrt(3)}, -1) -- (0, -2) -- ({sqrt(3)}, -1) -- ({sqrt(3)}, 1);
\draw [yshift = 260] ({6+sqrt(3)}, 1) -- (6+0, 2) -- ({6-sqrt(3)}, 1) -- ({6-sqrt(3)}, -1) -- (6+0, -2) -- ({6+sqrt(3)}, -1) -- ({6+sqrt(3)}, 1);
\draw [yshift = 260]({12+sqrt(3)}, 1) -- (12+0, 2) -- ({12-sqrt(3)}, 1) -- ({12-sqrt(3)}, -1) -- (12+0, -2) -- ({12+sqrt(3)}, -1) -- ({12+sqrt(3)}, 1);

\draw [yshift = 260] ({-4.8+sqrt(3)}, 1) -- ({-sqrt(3)}, 1);
\draw [yshift = 260] ({-4.8+sqrt(3)}, -1) -- ({-sqrt(3)}, -1);
\draw [yshift = 260] ({sqrt(3)}, 1) -- ({6-sqrt(3)}, 1);
\draw [yshift = 260] ({sqrt(3)}, -1) -- ({6-sqrt(3)}, -1);
\draw [yshift = 260] ({6+sqrt(3)}, 1) -- ({12-sqrt(3)}, 1);
\draw [yshift = 260] ({6+sqrt(3)}, -1) -- ({12-sqrt(3)}, -1);
\draw [yshift = 260] ({12+sqrt(3)}, 1) -- ({16.8-sqrt(3)}, 1);
\draw [yshift = 260] ({12+sqrt(3)}, -1) -- ({16.8-sqrt(3)}, -1);

%%%%
\draw [xshift = 86, yshift = 130] ({-sqrt(3)}, 1) -- (-3+0, 2.6);
\draw [xshift = 86, yshift = 130](0, 2) -- (-1.25, 3.6);
\draw [xshift = 86, yshift = 130](0, 2) -- (1.3, 3.6);
\draw [xshift = 86, yshift = 130]({sqrt(3)}, 1) -- (3.05, 2.6);

\draw [xshift = 86, yshift = 130]({6-sqrt(3)}, 1) -- (3+0, 2.6);
\draw [xshift = 86, yshift = 130](6, 2) -- (4.75, 3.6);
\draw [xshift = 86, yshift = 130](6, 2) -- (7.3, 3.6);
\draw [xshift = 86, yshift = 130]({6+sqrt(3)}, 1) -- (9.05, 2.6);

\draw [xshift = 86, yshift = 130]({12-sqrt(3)}, 1) -- (9, 2.6);
\draw [xshift = 86, yshift = 130](12, 2) -- (10.75, 3.6);
\draw [yshift = 130](-3, 2) -- (-1.7, 3.6);
\draw [yshift = 130]({-3+sqrt(3)}, 1) -- (.05, 2.6);

%%%%

\draw [xshift = 86, yshift = -130] ({-.8-sqrt(3)}, 2) -- (-3+0, 2.6);
\draw [xshift = 86, yshift = -130](-.8, 3) -- (-1.3, 3.6);
\draw [yshift = -130](-2.2, 3) -- (-1.7, 3.6);
\draw [yshift = -130]({-2.2+sqrt(3)}, 2) -- (.05, 2.6);

\draw [xshift = 256, yshift = -130] ({-.8-sqrt(3)}, 2) -- (-3+0, 2.6);
\draw [xshift = 256, yshift = -130](-.8, 3) -- (-1.3, 3.6);
\draw [xshift = 170,yshift = -130](-2.2, 3) -- (-1.7, 3.6);
\draw [xshift = 170,yshift = -130]({-2.2+sqrt(3)}, 2) -- (.05, 2.6);

\draw [xshift = 426, yshift = -130] ({-.8-sqrt(3)}, 2) -- (-3+0, 2.6);
\draw [xshift = 426, yshift = -130](-.8, 3) -- (-1.3, 3.6);
\draw [xshift = 340,yshift = -130](-2.2, 3) -- (-1.7, 3.6);
\draw [xshift = 340,yshift = -130]({-2.2+sqrt(3)}, 2) -- (.05, 2.6);

%%%%
\draw [yshift = 260] ({-sqrt(3)}, 1) -- (-2.2+0, 1.6);
\draw [yshift = 260] (0, 2) -- (-0.5, 2.6);
\draw [yshift = 260] (0, 2) -- (.5, 2.6);
\draw [yshift = 260] ({sqrt(3)}, 1) -- (2.2, 1.6);

\draw [xshift = 170, yshift = 260] ({-sqrt(3)}, 1) -- (-2.2+0, 1.6);
\draw [xshift = 170, yshift = 260] (0, 2) -- (-0.5, 2.6);
\draw [xshift = 170, yshift = 260] (0, 2) -- (.5, 2.6);
\draw [xshift = 170, yshift = 260] ({sqrt(3)}, 1) -- (2.2, 1.6);

\draw [xshift = 340, yshift = 260] ({-sqrt(3)}, 1) -- (-2.2+0, 1.6);
\draw [xshift = 340, yshift = 260] (0, 2) -- (-0.5, 2.6);
\draw [xshift = 340, yshift = 260] (0, 2) -- (.5, 2.6);
\draw [xshift = 340, yshift = 260] ({sqrt(3)}, 1) -- (2.2, 1.6);

\node[ver] () at (-.85,-.8){$u_{0,-2}$};
\node[ver] () at (0,-1.4){$v_{0,-2}$};
\node[ver] () at (.85,-.8){$w_{0,-2}$};
\node[ver] () at (.85,.6){$u_{0,-1}$};
\node[ver] () at (0,1.3){$v_{0,-1}$};
\node[ver] () at (-.85,.6){$w_{0,-1}$};

\node[ver] () at (-.85+6,-.8){$u_{1,-2}$};
\node[ver] () at (0+6,-1.4){$v_{1,-2}$};
\node[ver] () at (.85+6,-.8){$w_{1,-2}$};
\node[ver] () at (.85+6,.6){$u_{1,-1}$};
\node[ver] () at (0+6,1.3){$v_{1,-1}$};
\node[ver] () at (-.85+6,.6){$w_{1,-1}$};

\node[ver] () at (-.85+12,-.8){$u_{2,-2}$};
\node[ver] () at (0+12,-1.4){$v_{2,-2}$};
\node[ver] () at (.85+12,-.8){$w_{2,-2}$};
\node[ver] () at (.85+12,.6){$u_{2,-1}$};
\node[ver] () at (0+12,1.3){$v_{2,-1}$};
\node[ver] () at (-.85+12,.6){$w_{2,-1}$};
%%%%%%
%\node[ver] () at (-.8-3,-.8+4.5){ $u_{0,0}$};
\node[ver] () at (0-3,-1.5+4.6){$v_{-1,0}$};
\node[ver] () at (.7-3,-.8+4.5){$w_{-1,0}$};
\node[ver] () at (.7-3,.6+4.5){$u_{-1,1}$};
\node[ver] () at (0-3,1.3+4.6){$v_{-1,1}$};
%\node[ver] () at (-.8-3,.6+4.5){ $w_{0,1}$};

\node[ver] () at (-1+3,-.8+4.5){$u_{0,0}$};
\node[ver] () at (0+3,-1.5+4.6){$v_{0,0}$};
\node[ver] () at (1+3,-.8+4.5){$w_{0,0}$};
\node[ver] () at (5.5,.6+4.5){$u_{0,1}$};
\node[ver] () at (0+3,1.3+4.6){$v_{0,1}$};
\node[ver] () at (.6,.6+4.5){$w_{0,1}$};

\node[ver] () at (-1+9,-.8+4.5){$u_{1,0}$};
\node[ver] () at (0+9,-1.5+4.6){$v_{1,0}$};
\node[ver] () at (1+9,-.8+4.5){$w_{1,0}$};
\node[ver] () at (1+9,.6+4.5){$u_{1,1}$};
\node[ver] () at (0+9,1.3+4.6){$v_{1,1}$};
\node[ver] () at (-1+9,.6+4.5){$w_{1,1}$};

\node[ver] () at (-1+15,-.8+4.5){$u_{2,0}$};
\node[ver] () at (0+15,-1.5+4.6){$v_{2,0}$};
%\node[ver] () at (.9+15,-.8+4.5){ $w_{2,0}$};
%\node[ver] () at (.9+15,.6+4.5){ $u_{2,1}$};
\node[ver] () at (0+15,1.3+4.6){$v_{2,1}$};
\node[ver] () at (-1+15,.6+4.5){$w_{2,1}$};
%%%%%
\node[ver] () at (-.85,-.8+9){$u_{-1,2}$};
\node[ver] () at (-1,7){$v_{-1,2}$};
\node[ver] () at (2.7,-.8+9.2){$w_{-1,2}$};
\node[ver] () at (.85,.6+9){$u_{-1,3}$};
\node[ver] () at (0,1.3+9.1){$v_{-1,3}$};
\node[ver] () at (-.85,.6+9){$w_{-1,3}$};

\node[ver] () at (-.85+5.7,-.8+9.2){$u_{0,2}$};
\node[ver] () at (0+6,-1.5+9.1){$v_{0,2}$};
\node[ver] () at (.85+6,-.8+9){$w_{0,2}$};
\node[ver] () at (.85+6,.6+9){$u_{0,3}$};
\node[ver] () at (0+6,1.3+9.1){$v_{0,3}$};
\node[ver] () at (-.85+6,.6+9){$w_{0,3}$};

\node[ver] () at (-1+12,-.8+9){$u_{1,2}$};
\node[ver] () at (0+12,-1.5+9.1){$v_{1,2}$};
\node[ver] () at (.85+12,-.8+9){$w_{1,2}$};
\node[ver] () at (.85+12,.6+9){$u_{1,3}$};
\node[ver] () at (0+12,1.3+9.1){$v_{1,3}$};
\node[ver] () at (-1+12,.6+9){$w_{1,3}$};

\draw [thick, dotted] (3,4.5) -- (9,4.5);
\draw [thick, dotted] (3,4.5) -- (6,9);
\draw [thick, dotted] (3,4.5) -- (0,9);

%\draw [very thick, dots=5 per 1cm] (3,4.5) -- (9,4.5);

%\draw [dashed] (3,4.5) -- (9,4.5);
%\draw [dashed] (3,4.5) -- (6,9);
%\draw [dashed] (3,4.5) -- (0,9);

\node[ver] () at (3,4.5){$\bullet$};
\node[ver] () at (9,4.5){$\bullet$};
\node[ver] () at (6,9){$\bullet$};
\node[ver] () at (0,9){$\bullet$};

\node[ver] () at (5.2, -4){\normalsize (b): Rhombitrihexagonal tiling $E_5$};

\end{tikzpicture}
\vspace{-2mm}
\caption*{Figure 3}
\end{figure}
%%%%%%%%%%%%%%%%%%%%%E4&E5%%%%%%%%%%%%%%%

%\smallskip 

\noindent {\sf Cases 4, 5, 6.} By the same arguments  the result follows when the vertex-type of $X$ is $[3^1,6^1,3^1,6^1]$, $[3^1,4^1,6^1,4^1]$ or $[3^1,12^2]$. Here, $X$ is a quotient of $E_4$, $E_5$ or $E_6$ (given in Figures 3 and 4). (Observe that vertices of $E_6$ also form six ${\mathcal H}_6$-orbits.)

\smallskip 

\noindent {\sf Case 7.} Finally, assume that the vertex-type of $X$ is $[4^1,6^1,12^1]$.  By Proposition \ref{prop:Archi}, we can assume that $X = E_{7}/{\mathcal K}_{7}$ for some subgroup ${\mathcal K}_{7}$ of ${\rm Aut}(E_{7})$. By the similar arguments, ${\mathcal K}_{7}$ is a subgroup of ${\mathcal H}_{7}$ and  $\eta_{7} : E_{7}\to X$ is a covering. 
We take $(0,0)$ as the middle point of $u_{0,0}$ and $u_{0,1}$ of $E_7$ and  $A_7 = u_{1,0} - u_{0,0}$, $B_7 = u_{0,2}- u_{0,0}$,   $F_7 = u_{-1,2}- u_{0,0}$  (see Fig. 4 (b)). Then $A_7 + F_7 = B_7$ and hence ${\mathcal H}_7 = \langle \alpha_7 : z \mapsto z + A_7, \, \beta_7 : z \mapsto z + B_7\rangle$. By similar argument as in Claim 3, there exists a factor $m$ of $|ad-bc|$ such that ${\mathcal L}_7 := \langle \alpha_7^m, \beta_7^m \rangle \leq {\mathcal K}_7$. 
 
Let $\rho_7$ be the function obtained by  60 degrees anticlockwise rotation and let $\tau$ be the reflection 
with respect to the line along the vector $A_7$. Then $\rho_7, \tau_7\in {\rm Aut}(E_{3})$ and 
$\rho_7(A_7) = B_7$, $\rho_7(B_7) = F_7$,  $\rho_7(F_7) = - A_7$, $\tau_7(B_7) = -F_7$,  $\rho_7(F_7) = - B_7$.  Consider the group 
\begin{align*}
  {\mathcal G}_{7} & := \langle \alpha_7, \beta_7, \rho_7, \tau_7 \rangle.
\end{align*}

Clearly,  vertices of $E_{7}$ form twelve ${\mathcal H}_{7}$-orbits (represented by vertices of any $12$-gon). Therefore, vertices of $E_7$ form two $\langle \alpha_7, \beta_7, \rho_7 \rangle$-orbits. These are 
$\mbox{orb}(x_{0,0})$ and  $\mbox{orb}(u_{0,1})$   (see Fig. 4 (b)). 
Since $\tau_7(u_{1,0})= x_{0,0}$, it follows that  ${\mathcal G}_7$ acts vertex-transitively on $E_7$.  

%%%%%%%%%%%%%%%%%%%%%%%E6&E7%%%%%%%%%%%%%%%%%%%%
\begin{figure}[ht!]
\tiny
\tikzstyle{ver}=[]
\tikzstyle{vert}=[circle, draw, fill=black!100, inner sep=0pt, minimum width=4pt]
\tikzstyle{vertex}=[circle, draw, fill=black!00, inner sep=0pt, minimum width=4pt]
\tikzstyle{edge} = [draw,thick,-]
\centering
%%%%%%%%%%%%%%%%%%%%%%%%E6%%%%%%%%%%%%%%%%%%%%
\begin{tikzpicture}[scale=0.45]
%\begin{scope}[shift={(-50,13)}]
\draw [xshift = -55, yshift = 285] ({2*cos(195)},{2*sin(185)}) -- ({2*cos(195)},{2*sin(195)}) -- ({2*cos(225)},{2*sin(225)}) -- ({2*cos(255)},{2*sin(255)}) -- ({2*cos(285)},{2*sin(285)}) -- ({2*cos(315)},{2*sin(315)}) -- ({2*cos(345)},{2*sin(345)});

\draw [xshift = 55, yshift = 285] ({2*cos(195)},{2*sin(185)}) -- ({2*cos(195)},{2*sin(195)}) -- ({2*cos(225)},{2*sin(225)}) -- ({2*cos(255)},{2*sin(255)}) -- ({2*cos(285)},{2*sin(285)}) -- ({2*cos(315)},{2*sin(315)}) -- ({2*cos(345)},{2*sin(345)});

\draw [xshift = 165, yshift = 285] ({2*cos(195)},{2*sin(185)}) -- ({2*cos(195)},{2*sin(195)}) -- ({2*cos(225)},{2*sin(225)}) -- ({2*cos(255)},{2*sin(255)}) -- ({2*cos(285)},{2*sin(285)}) -- ({2*cos(315)},{2*sin(315)}) -- ({2*cos(345)},{2*sin(345)});

\draw [xshift = 275, yshift = 285] ({2*cos(195)},{2*sin(185)}) -- ({2*cos(195)},{2*sin(195)}) -- ({2*cos(225)},{2*sin(225)}) -- ({2*cos(255)},{2*sin(255)}) -- ({2*cos(270)},{2*sin(270)});

%%%%%

\draw [xshift = -110] ({2*cos(90)},{2*sin(90)}) -- ({2*cos(75)},{2*sin(75)}) -- ({2*cos(45)},{2*sin(45)}) -- ({2*cos(15)},{2*sin(15)}) -- ({2*cos(345)},{2*sin(345)}) -- ({2*cos(315)},{2*sin(315)}) -- ({2*cos(285)},{2*sin(285)}) -- ({2*cos(270)},{2*sin(270)});

\draw ({2*cos(15)},{2*sin(15)}) -- ({2*cos(45)},{2*sin(45)}) -- ({2*cos(75)},{2*sin(75)}) --  ({2*cos(105)},{2*sin(105)}) -- ({2*cos(135)},{2*sin(135)}) -- ({2*cos(165)},{2*sin(165)}) -- ({2*cos(195)},{2*sin(195)}) -- ({2*cos(225)},{2*sin(225)}) -- ({2*cos(255)},{2*sin(255)}) -- ({2*cos(285)},{2*sin(285)}) -- ({2*cos(315)},{2*sin(315)}) -- ({2*cos(345)},{2*sin(345)}) -- ({2*cos(15)},{2*sin(15)});

\draw [xshift = 110] ({2*cos(15)},{2*sin(15)}) -- ({2*cos(45)},{2*sin(45)}) -- ({2*cos(75)},{2*sin(75)}) --  ({2*cos(105)},{2*sin(105)}) -- ({2*cos(135)},{2*sin(135)}) -- ({2*cos(165)},{2*sin(165)}) -- ({2*cos(195)},{2*sin(195)}) -- ({2*cos(225)},{2*sin(225)}) -- ({2*cos(255)},{2*sin(255)}) -- ({2*cos(285)},{2*sin(285)}) -- ({2*cos(315)},{2*sin(315)}) -- ({2*cos(345)},{2*sin(345)}) -- ({2*cos(15)},{2*sin(15)});

\draw [xshift = 220] ({2*cos(15)},{2*sin(15)}) -- ({2*cos(45)},{2*sin(45)}) -- ({2*cos(75)},{2*sin(75)}) --  ({2*cos(105)},{2*sin(105)}) -- ({2*cos(135)},{2*sin(135)}) -- ({2*cos(165)},{2*sin(165)}) -- ({2*cos(195)},{2*sin(195)}) -- ({2*cos(225)},{2*sin(225)}) -- ({2*cos(255)},{2*sin(255)}) -- ({2*cos(285)},{2*sin(285)}) -- ({2*cos(315)},{2*sin(315)}) -- ({2*cos(345)},{2*sin(345)}) -- ({2*cos(15)},{2*sin(15)});

%%%%%%%%%

\draw [xshift = -55, yshift = 95] ({2*cos(15)},{2*sin(15)}) -- ({2*cos(45)},{2*sin(45)}) -- ({2*cos(75)},{2*sin(75)}) --  ({2*cos(105)},{2*sin(105)}) -- ({2*cos(135)},{2*sin(135)}) -- ({2*cos(165)},{2*sin(165)}) -- ({2*cos(195)},{2*sin(195)}) -- ({2*cos(225)},{2*sin(225)}) -- ({2*cos(255)},{2*sin(255)}) -- ({2*cos(285)},{2*sin(285)}) -- ({2*cos(315)},{2*sin(315)}) -- ({2*cos(345)},{2*sin(345)}) -- ({2*cos(15)},{2*sin(15)});

\draw [xshift = 55, yshift = 95] ({2*cos(15)},{2*sin(15)}) -- ({2*cos(45)},{2*sin(45)}) -- ({2*cos(75)},{2*sin(75)}) --  ({2*cos(105)},{2*sin(105)}) -- ({2*cos(135)},{2*sin(135)}) -- ({2*cos(165)},{2*sin(165)}) -- ({2*cos(195)},{2*sin(195)}) -- ({2*cos(225)},{2*sin(225)}) -- ({2*cos(255)},{2*sin(255)}) -- ({2*cos(285)},{2*sin(285)}) -- ({2*cos(315)},{2*sin(315)}) -- ({2*cos(345)},{2*sin(345)}) -- ({2*cos(15)},{2*sin(15)});

\draw [xshift = 165, yshift = 95] ({2*cos(15)},{2*sin(15)}) -- ({2*cos(45)},{2*sin(45)}) -- ({2*cos(75)},{2*sin(75)}) --  ({2*cos(105)},{2*sin(105)}) -- ({2*cos(135)},{2*sin(135)}) -- ({2*cos(165)},{2*sin(165)}) -- ({2*cos(195)},{2*sin(195)}) -- ({2*cos(225)},{2*sin(225)}) -- ({2*cos(255)},{2*sin(255)}) -- ({2*cos(285)},{2*sin(285)}) -- ({2*cos(315)},{2*sin(315)}) -- ({2*cos(345)},{2*sin(345)}) -- ({2*cos(15)},{2*sin(15)});

\draw [xshift = 275, yshift = 95] ({2*cos(90)},{2*sin(90)}) -- ({2*cos(105)},{2*sin(105)}) -- ({2*cos(135)},{2*sin(135)}) -- ({2*cos(165)},{2*sin(165)}) -- ({2*cos(195)},{2*sin(195)}) -- ({2*cos(225)},{2*sin(225)}) -- ({2*cos(255)},{2*sin(255)}) -- ({2*cos(270)},{2*sin(270)});

%%%%%%%%%%%%

\draw [xshift = -110, yshift = 190] ({2*cos(90)},{2*sin(90)}) -- ({2*cos(75)},{2*sin(75)}) -- ({2*cos(45)},{2*sin(45)}) -- ({2*cos(15)},{2*sin(15)}) -- ({2*cos(345)},{2*sin(345)}) -- ({2*cos(315)},{2*sin(315)}) -- ({2*cos(285)},{2*sin(285)}) -- ({2*cos(270)},{2*sin(270)});

\draw [yshift = 190] ({2*cos(15)},{2*sin(15)}) -- ({2*cos(45)},{2*sin(45)}) -- ({2*cos(75)},{2*sin(75)}) --  ({2*cos(105)},{2*sin(105)}) -- ({2*cos(135)},{2*sin(135)}) -- ({2*cos(165)},{2*sin(165)}) -- ({2*cos(195)},{2*sin(195)}) -- ({2*cos(225)},{2*sin(225)}) -- ({2*cos(255)},{2*sin(255)}) -- ({2*cos(285)},{2*sin(285)}) -- ({2*cos(315)},{2*sin(315)}) -- ({2*cos(345)},{2*sin(345)}) -- ({2*cos(15)},{2*sin(15)});

\draw [xshift = 110, yshift = 190] ({2*cos(15)},{2*sin(15)}) -- ({2*cos(45)},{2*sin(45)}) -- ({2*cos(75)},{2*sin(75)}) --  ({2*cos(105)},{2*sin(105)}) -- ({2*cos(135)},{2*sin(135)}) -- ({2*cos(165)},{2*sin(165)}) -- ({2*cos(195)},{2*sin(195)}) -- ({2*cos(225)},{2*sin(225)}) -- ({2*cos(255)},{2*sin(255)}) -- ({2*cos(285)},{2*sin(285)}) -- ({2*cos(315)},{2*sin(315)}) -- ({2*cos(345)},{2*sin(345)}) -- ({2*cos(15)},{2*sin(15)});

\draw [xshift = 220, yshift = 190] ({2*cos(15)},{2*sin(15)}) -- ({2*cos(45)},{2*sin(45)}) -- ({2*cos(75)},{2*sin(75)}) --  ({2*cos(105)},{2*sin(105)}) -- ({2*cos(135)},{2*sin(135)}) -- ({2*cos(165)},{2*sin(165)}) -- ({2*cos(195)},{2*sin(195)}) -- ({2*cos(225)},{2*sin(225)}) -- ({2*cos(255)},{2*sin(255)}) -- ({2*cos(285)},{2*sin(285)}) -- ({2*cos(315)},{2*sin(315)}) -- ({2*cos(345)},{2*sin(345)}) -- ({2*cos(15)},{2*sin(15)});

\draw [xshift = -55, yshift = -95] ({2*cos(15)},{2*sin(15)}) -- ({2*cos(45)},{2*sin(45)}) -- ({2*cos(75)},{2*sin(75)}) --  ({2*cos(105)},{2*sin(105)}) -- ({2*cos(135)},{2*sin(135)}) -- ({2*cos(165)},{2*sin(165)}) -- ({2*cos(165)},{2*sin(172.5)});

\draw [xshift = 55, yshift = -95] ({2*cos(15)},{2*sin(15)}) -- ({2*cos(45)},{2*sin(45)}) -- ({2*cos(75)},{2*sin(75)}) --  ({2*cos(105)},{2*sin(105)}) -- ({2*cos(135)},{2*sin(135)}) -- ({2*cos(165)},{2*sin(165)}) -- ({2*cos(165)},{2*sin(172.5)});

\draw [xshift = 165, yshift = -95] ({2*cos(15)},{2*sin(15)}) -- ({2*cos(45)},{2*sin(45)}) -- ({2*cos(75)},{2*sin(75)}) --  ({2*cos(105)},{2*sin(105)}) -- ({2*cos(135)},{2*sin(135)}) -- ({2*cos(165)},{2*sin(165)}) -- ({2*cos(165)},{2*sin(172.5)});

\draw [xshift = 275, yshift = -95] ({2*cos(90)},{2*sin(90)}) -- ({2*cos(105)},{2*sin(105)}) -- ({2*cos(135)},{2*sin(135)}) -- ({2*cos(165)},{2*sin(165)}) -- ({2*cos(165)},{2*sin(172.5)});

%\node[ver] () at (1.1-4.4,-1.3){\tiny $w_{-1,-1}$};
\node[ver] () at (-2.3,-2.2){\tiny $w_{-2,-1}$};
\node[ver] () at (1.1-4.2,1.2){\tiny $v_{-1,0}$};
\node[ver] () at (.5-4.3,1.7){\tiny $v_{-2,0}$};

\node[ver] () at (-.9,-.6){\tiny $u_{-1,-1}$};
\node[ver] () at (-.8,-1.3){\tiny $v_{0,-1}$};
\node[ver] () at (0,-1.7){\tiny $v_{1,-1}$};
\node[ver] () at (2.7,-.6){\tiny $u_{1,-1}$};
\node[ver] () at (.8,-1.3){\tiny $w_{1,-1}$};
\node[ver] () at (1.3,-2.2){\tiny $w_{0,-1}$};
\node[ver] () at (2.6,.4){\tiny $u_{0,0}$};
\node[ver] () at (1,1.2){\tiny $v_{1,0}$};
\node[ver] () at (1.1,1.9){\tiny $v_{0,0}$};
\node[ver] () at (-1,.4){\tiny $u_{-2,0}$};
\node[ver] () at (-.5,1.2){\tiny $w_{-2,0}$};
\node[ver] () at (-1.4,1.9){\tiny $w_{-1,0}$};

%\node[ver] () at (-1.4,-.5){\tiny $u_{0,-2}$};
\node[ver] () at (3.2,-1.3){\tiny $v_{2,-1}$};
\node[ver] () at (3.8,-1.7){\tiny $v_{3,-1}$};
\node[ver] () at (5.1,-.5){\tiny $u_{3,-1}$};
\node[ver] () at (1.1+3.6,-1.3){\tiny $w_{3,-1}$};
\node[ver] () at (1.3+3.8,-2.2){\tiny $w_{2,-1}$};
\node[ver] () at (1.4+3.8,.5){\tiny $u_{2,0}$};
\node[ver] () at (1.1+3.8,1.1){\tiny $v_{3,0}$};
\node[ver] () at (5,1.9){\tiny $v_{2,0}$};
%\node[ver] () at (-1.4,.5){\tiny $u_{0,-1}$};
\node[ver] () at (-.7+3.8,1.3){\tiny $w_{0,0}$};
\node[ver] () at (2.7,1.9){\tiny $w_{1,0}$};

%\node[ver] () at (-1.4,-.5){\tiny $u_{0,-2}$};
\node[ver] () at (6.9,-1.3){\tiny $v_{4,-1}$};
\node[ver] () at (6.7,-2.2){\tiny $v_{5,-1}$};
\node[ver] () at (1.4+7.6,-.5){\tiny $u_{5,-1}$};
\node[ver] () at (8.4,-1.3){\tiny $w_{5,-1}$};
\node[ver] () at (1.3+7.6,-2.2){\tiny $w_{4,-1}$};
\node[ver] () at (1.5+7.6,.5){\tiny $u_{4,0}$};
\node[ver] () at (1.1+7.6,1.1){\tiny $v_{5,0}$};
\node[ver] () at (8.8,1.9){\tiny $v_{4,0}$};
%\node[ver] () at (-1.4,.5){\tiny $u_{0,-1}$};
\node[ver] () at (-.7+7.6,1.3){\tiny $w_{2,0}$};
\node[ver] () at (-.4+7,1.9){\tiny $w_{3,0}$};

%%%%%%%%%%%%
%\node[ver] () at (1.1-4.2,-1.2 + 5){\tiny $u_{-1,-1}$};
\node[ver] () at (1-4.1,-1.2 + 4){\tiny $u_{-3,0}$};

%\node[ver] () at (.9,-1.2 + 5){\tiny $u_{-1,-1}$};
\node[ver] () at (0.8,-1.2 + 4){\tiny $u_{-1,0}$};

%\node[ver] () at (4.6,-1.2 + 5){\tiny $u_{-1,-1}$};
\node[ver] () at (4.5,-1.2 + 4){\tiny $u_{1,0}$};

%\node[ver] () at (8.5,-1.2 + 5){\tiny $u_{-1,-1}$};
\node[ver] () at (8.3,-1.2 + 4){\tiny $u_{3,0}$};

\node[ver] () at (1.2-4,-1.2 - 1.7){\tiny $u_{-2,-1}$};
\node[ver] () at (1.4-.4,-1.2 - 1.7){\tiny $u_{0,-1}$};
\node[ver] () at (1.4+3.4,-1.2 - 1.7){\tiny $u_{2,-1}$};
\node[ver] () at (1.6+7,-1.2 - 1.7){\tiny $u_{4,-1}$};
%%%%%%%%

\node[ver] () at (-3.3,5.3){\tiny $w_{-3,1}$};
\node[ver] () at (-2.5,4.6){\tiny $w_{-4,1}$};
\node[ver] () at (1.1-4.2,7.8){\tiny $v_{-3,2}$};
\node[ver] () at (.5-4.3,1.7+6.6){\tiny $v_{-4,2}$};

\node[ver] () at (-1.2,-.5+6.7){\tiny $u_{-3,1}$};
\node[ver] () at (-.7,5.3){\tiny $v_{-2,1}$};
\node[ver] () at (-1.2,4.6){\tiny $v_{-1,1}$};
\node[ver] () at (1.2,-.5+6.7){\tiny $u_{-1,1}$};
\node[ver] () at (.8,5.3){\tiny $w_{-1,1}$};
\node[ver] () at (1.4,4.6){\tiny $w_{-2,1}$};
\node[ver] () at (1.1,.5+6.7){\tiny $u_{-2,2}$};
\node[ver] () at (.8,7.8){\tiny $v_{-1,2}$};
\node[ver] () at (1.2,8.7){\tiny $v_{-2,2}$};
\node[ver] () at (-2.8,.5+6.7){\tiny $u_{-4,2}$};
\node[ver] () at (-.6,7.8){\tiny $w_{-4,2}$};
\node[ver] () at (-1.2,8.7){\tiny $w_{-3,2}$};

%\node[ver] () at (-1.4,-.5){\tiny $u_{0,-2}$};
\node[ver] () at (2.1,5){\tiny $v_{0,1}$};
\node[ver] () at (3.7,4.9){\tiny $v_{1,1}$};
\node[ver] () at (1.4+3.8,-.5+6.7){\tiny $u_{1,1}$};
\node[ver] () at (1.1+3.8,-1.1+6.7){\tiny $w_{1,1}$};
\node[ver] () at (1.3+3.8,-1.6+6.2){\tiny $w_{0,1}$};
\node[ver] () at (1.4+3.8,.5+6.7){\tiny $u_{0,2}$};
\node[ver] () at (1.1+3.8,7.8){\tiny $v_{1,2}$};
\node[ver] () at (4.8,8.7){\tiny $v_{0,2}$};
%\node[ver] () at (-1.4,.5){\tiny $u_{0,-1}$};
\node[ver] () at (3.4,7.8){\tiny $w_{-2,2}$};
\node[ver] () at (3.8,8.2){\tiny $w_{-1,2}$};

%\node[ver] () at (-1.4,-.5){\tiny $u_{0,-2}$};
\node[ver] () at (-1+7.6,-1.1+6.7){\tiny $v_{2,1}$};
\node[ver] () at (-.5+7.6,-1.6+6.7){\tiny $v_{3,1}$};
\node[ver] () at (1.4+7.6,-.5+6.7){\tiny $u_{3,1}$};
\node[ver] () at (1.1+7.6,-1.1+6.7){\tiny $w_{3,1}$};
\node[ver] () at (.5+7.6,-1.6+6.7){\tiny $w_{2,1}$};
\node[ver] () at (1.5+7.6,.5+6.7){\tiny $u_{2,2}$};
\node[ver] () at (1.1+7.6,7.8){\tiny $v_{3,2}$};
\node[ver] () at (8.7,8.7){\tiny $v_{2,2}$};
%\node[ver] () at (-1.4,.5){\tiny $u_{0,-1}$};
\node[ver] () at (-.7+7.6,7.8){\tiny $w_{0,2}$};
\node[ver] () at (6.5,8.7){\tiny $w_{1,2}$};

%%%%%%%%%%%%
%\node[ver] () at (1.1-4.2,-1.2 + 5){\tiny $u_{-1,-1}$};
\node[ver] () at (1-4.1,-1.2 + 4+6.7){\tiny $u_{-5,2}$};

%\node[ver] () at (.9,-1.2 + 5){\tiny $u_{-1,-1}$};
\node[ver] () at (0.8,-1.2 + 4+6.7){\tiny $u_{-3,2}$};

%\node[ver] () at (4.6,-1.2 + 5){\tiny $u_{-1,-1}$};
\node[ver] () at (4.6,-1.2 + 4+6.7){\tiny $u_{-1,2}$};

%\node[ver] () at (8.5,-1.2 + 5){\tiny $u_{-1,-1}$};
\node[ver] () at (8.4,-1.2 + 4+6.7){\tiny $u_{1,2}$};
%%%
%\node[ver] () at (1.1-4.2,-1.2 + 10.5){\tiny $u_{-1,-1}$};
%\node[ver] () at (1.2-.4,-1.2 + 10.5){\tiny $u_{-1,-1}$};
%\node[ver] () at (1.3+3.4,-1.2 + 10.5){\tiny $u_{-1,-1}$};
%\node[ver] () at (1.5+7,-1.2 + 10.5){\tiny $u_{-1,-1}$};

\node[ver] () at (1.1-4.2,-1.2 - 1.7+6.7){\tiny $u_{-4,1}$};
\node[ver] () at (1.2-.4,-1.2 - 1.7+6.7){\tiny $u_{-2,1}$};
\node[ver] () at (1+3.5,-1.2 - 1.7+6.7){\tiny $u_{0,1}$};
\node[ver] () at (1.4+7,-1.2 - 1.7+6.7){\tiny $u_{2,1}$};
%%%%%%%%

\draw [thick, dotted] (2,3.3) -- (5.8,3.3);
\draw [thick, dotted] (2,3.3) -- (3.7,6.7);
\draw [thick, dotted] (2,3.3) -- (0,6.7);

\node[ver] () at (2,3.3){$\bullet$};
\node[ver] () at (5.8,3.3){$\bullet$};
\node[ver] () at (3.7,6.7){$\bullet$};
\node[ver] () at (0,6.7){$\bullet$};

\node[ver] () at (3, -4.5){\normalsize (a): Truncated hexagonal tiling $E_{6}$};

\end{tikzpicture}
%%%%%%%%%%%%%%%%%%%%%%E7%%%%%%%%%%%%%%%%%%%
\begin{tikzpicture}[scale=0.45]

\draw [xshift = -80, yshift = 8] ({1*cos(25)},{1*sin(14)}) -- ({1*cos(100)},{1*sin(14)});
\draw [xshift = -80, yshift = -22] ({1*cos(25)},{1*sin(14)}) -- ({1*cos(100)},{1*sin(14)});

\draw [xshift = -140] ({2*cos(90)},{2*sin(90)}) -- ({2*cos(75)},{2*sin(75)}) -- ({2*cos(45)},{2*sin(45)}) -- ({2*cos(15)},{2*sin(15)}) -- ({2*cos(345)},{2*sin(345)}) -- ({2*cos(315)},{2*sin(315)}) -- ({2*cos(285)},{2*sin(285)}) -- ({2*cos(270)},{2*sin(270)});

\draw [xshift = 60, yshift = 8] ({1*cos(25)},{1*sin(14)}) -- ({1*cos(100)},{1*sin(14)});
\draw [xshift = 60, yshift = -22] ({1*cos(25)},{1*sin(14)}) -- ({1*cos(100)},{1*sin(14)});

\draw ({2*cos(15)},{2*sin(15)}) -- ({2*cos(45)},{2*sin(45)}) -- ({2*cos(75)},{2*sin(75)}) --  ({2*cos(105)},{2*sin(105)}) -- ({2*cos(135)},{2*sin(135)}) -- ({2*cos(165)},{2*sin(165)}) -- ({2*cos(195)},{2*sin(195)}) -- ({2*cos(225)},{2*sin(225)}) -- ({2*cos(255)},{2*sin(255)}) -- ({2*cos(285)},{2*sin(285)}) -- ({2*cos(315)},{2*sin(315)}) -- ({2*cos(345)},{2*sin(345)}) -- ({2*cos(15)},{2*sin(15)});

\draw [xshift = -78, yshift = -22] (-1.7,2.7) -- (-1.2,3.6);
\draw [xshift = -90, yshift = 10] (-.35,1.05) -- (0.1,1.95);

\draw [xshift = 62, yshift = -22] (-1.7,2.7) -- (-1.2,3.6);
\draw [xshift = 50, yshift = 10] (-.35,1.05) -- (0.1,1.95);

\draw [xshift = 202, yshift = -22] (-1.7,2.7) -- (-1.2,3.6);
\draw [xshift = 190, yshift = 10] (-.35,1.05) -- (0.1,1.95);

\draw [xshift = 342, yshift = -22] (-1.7,2.7) -- (-1.2,3.6);
\draw [xshift = 330, yshift = 10] (-.35,1.05) -- (0.1,1.95);

\draw [xshift = 200, yshift = 8] ({1*cos(25)},{1*sin(14)}) -- ({1*cos(100)},{1*sin(14)});
\draw [xshift = 200, yshift = -22] ({1*cos(25)},{1*sin(14)}) -- ({1*cos(100)},{1*sin(14)});

\draw [xshift = 140] ({2*cos(15)},{2*sin(15)}) -- ({2*cos(45)},{2*sin(45)}) -- ({2*cos(75)},{2*sin(75)}) --  ({2*cos(105)},{2*sin(105)}) -- ({2*cos(135)},{2*sin(135)}) -- ({2*cos(165)},{2*sin(165)}) -- ({2*cos(195)},{2*sin(195)}) -- ({2*cos(225)},{2*sin(225)}) -- ({2*cos(255)},{2*sin(255)}) -- ({2*cos(285)},{2*sin(285)}) -- ({2*cos(315)},{2*sin(315)}) -- ({2*cos(345)},{2*sin(345)}) -- ({2*cos(15)},{2*sin(15)});

\draw [xshift = 280] ({2*cos(15)},{2*sin(15)}) -- ({2*cos(45)},{2*sin(45)}) -- ({2*cos(75)},{2*sin(75)}) --  ({2*cos(105)},{2*sin(105)}) -- ({2*cos(135)},{2*sin(135)}) -- ({2*cos(165)},{2*sin(165)}) -- ({2*cos(195)},{2*sin(195)}) -- ({2*cos(225)},{2*sin(225)}) -- ({2*cos(255)},{2*sin(255)}) -- ({2*cos(285)},{2*sin(285)}) -- ({2*cos(315)},{2*sin(315)}) -- ({2*cos(345)},{2*sin(345)}) -- ({2*cos(15)},{2*sin(15)});

\draw [xshift = 0, yshift = 50] (-.5,.15) -- (-1.13,1.05);
\draw [xshift = -26, yshift = 35] (-.5,.15) -- (-1.13,1.05);

\draw [xshift = 140, yshift = 50] (-.5,.15) -- (-1.13,1.05);
\draw [xshift = 114, yshift = 35] (-.5,.15) -- (-1.13,1.05);

\draw [xshift = 280, yshift = 50] (-.5,.15) -- (-1.13,1.05);
\draw [xshift = 254, yshift = 35] (-.5,.15) -- (-1.13,1.05);

%%%%%%%%%

\draw [xshift = -68, yshift = -70] (-.92,.7) -- (-1.13,1.05);
\draw [xshift = -93, yshift = -85] (-.92,.7) -- (-1.13,1.05);

\draw [xshift = 72, yshift = -70] (-.92,.7) -- (-1.13,1.05);
\draw [xshift = 47, yshift = -85] (-.92,.7) -- (-1.13,1.05);

\draw [xshift = 212, yshift = -70] (-.92,.7) -- (-1.13,1.05);
\draw [xshift = 187, yshift = -85] (-.92,.7) -- (-1.13,1.05);

\draw [xshift = 352, yshift = -70] (-.92,.7) -- (-1.13,1.05);
\draw [xshift = 327, yshift = -85] (-.92,.7) -- (-1.13,1.05);

%%%

\draw [xshift = -70, yshift = 170] (-.6,.15) -- (-1.13,1.05);
\draw [xshift = -96, yshift = 155] (-.6,.15) -- (-1.13,1.05);

\draw [xshift = 70, yshift = 170] (-.6,.15) -- (-1.13,1.05);
\draw [xshift = 44, yshift = 155] (-.6,.15) -- (-1.13,1.05);

\draw [xshift = 210, yshift = 170] (-.6,.15) -- (-1.13,1.05);
\draw [xshift = 184, yshift = 155] (-.6,.15) -- (-1.13,1.05);

\draw [xshift = 350, yshift = 170] (-.6,.15) -- (-1.13,1.05);
\draw [xshift = 324, yshift = 155] (-.6,.15) -- (-1.13,1.05);

\draw [xshift = -72, yshift = 120] ({2*cos(15)},{2*sin(15)}) -- ({2*cos(45)},{2*sin(45)}) -- ({2*cos(75)},{2*sin(75)}) --  ({2*cos(105)},{2*sin(105)}) -- ({2*cos(135)},{2*sin(135)}) -- ({2*cos(165)},{2*sin(165)}) -- ({2*cos(195)},{2*sin(195)}) -- ({2*cos(225)},{2*sin(225)}) -- ({2*cos(255)},{2*sin(255)}) -- ({2*cos(285)},{2*sin(285)}) -- ({2*cos(315)},{2*sin(315)}) -- ({2*cos(345)},{2*sin(345)}) -- ({2*cos(15)},{2*sin(15)});

\draw [xshift = -12, yshift = 128] ({1*cos(25)},{1*sin(14)}) -- ({1*cos(100)},{1*sin(14)});
\draw [xshift = -12, yshift = 98] ({1*cos(25)},{1*sin(14)}) -- ({1*cos(100)},{1*sin(14)});

\draw [xshift = 68, yshift = 120] ({2*cos(15)},{2*sin(15)}) -- ({2*cos(45)},{2*sin(45)}) -- ({2*cos(75)},{2*sin(75)}) --  ({2*cos(105)},{2*sin(105)}) -- ({2*cos(135)},{2*sin(135)}) -- ({2*cos(165)},{2*sin(165)}) -- ({2*cos(195)},{2*sin(195)}) -- ({2*cos(225)},{2*sin(225)}) -- ({2*cos(255)},{2*sin(255)}) -- ({2*cos(285)},{2*sin(285)}) -- ({2*cos(315)},{2*sin(315)}) -- ({2*cos(345)},{2*sin(345)}) -- ({2*cos(15)},{2*sin(15)});

\draw [xshift = 128, yshift = 128] ({1*cos(25)},{1*sin(14)}) -- ({1*cos(100)},{1*sin(14)});
\draw [xshift = 128, yshift = 98] ({1*cos(25)},{1*sin(14)}) -- ({1*cos(100)},{1*sin(14)});

\draw [xshift = 208, yshift = 120] ({2*cos(15)},{2*sin(15)}) -- ({2*cos(45)},{2*sin(45)}) -- ({2*cos(75)},{2*sin(75)}) --  ({2*cos(105)},{2*sin(105)}) -- ({2*cos(135)},{2*sin(135)}) -- ({2*cos(165)},{2*sin(165)}) -- ({2*cos(195)},{2*sin(195)}) -- ({2*cos(225)},{2*sin(225)}) -- ({2*cos(255)},{2*sin(255)}) -- ({2*cos(285)},{2*sin(285)}) -- ({2*cos(315)},{2*sin(315)}) -- ({2*cos(345)},{2*sin(345)}) -- ({2*cos(15)},{2*sin(15)});

\draw [xshift = 268, yshift = 128] ({1*cos(25)},{1*sin(14)}) -- ({1*cos(100)},{1*sin(14)});
\draw [xshift = 268, yshift = 98] ({1*cos(25)},{1*sin(14)}) -- ({1*cos(100)},{1*sin(14)});

\draw [xshift = 348, yshift = 120] ({2*cos(90)},{2*sin(90)}) -- ({2*cos(105)},{2*sin(105)}) -- ({2*cos(135)},{2*sin(135)}) -- ({2*cos(165)},{2*sin(165)}) -- ({2*cos(195)},{2*sin(195)}) -- ({2*cos(225)},{2*sin(225)}) -- ({2*cos(255)},{2*sin(255)}) -- ({2*cos(270)},{2*sin(270)});

\draw [xshift = 272, yshift = 98] (-1.7,2.7) -- (-1.25,3.58);
\draw [xshift = 259, yshift = 130] (-.35,1.05) -- (0.1,1.95);

\draw [xshift = 132, yshift = 98] (-1.7,2.7) -- (-1.25,3.58);
\draw [xshift = 119, yshift = 130] (-.35,1.05) -- (0.1,1.95);

\draw [xshift = -8, yshift = 98] (-1.7,2.7) -- (-1.25,3.58);
\draw [xshift = -21, yshift = 130] (-.35,1.05) -- (0.1,1.95);

%%%%%%%%%%%%

\draw [xshift = -83, yshift = 248] ({1*cos(25)},{1*sin(14)}) -- ({1*cos(100)},{1*sin(14)});
\draw [xshift = -83, yshift = 218] ({1*cos(25)},{1*sin(14)}) -- ({1*cos(100)},{1*sin(14)});

\draw [xshift = -143, yshift = 240] ({2*cos(90)},{2*sin(90)}) -- ({2*cos(75)},{2*sin(75)}) -- ({2*cos(45)},{2*sin(45)}) -- ({2*cos(15)},{2*sin(15)}) -- ({2*cos(345)},{2*sin(345)}) -- ({2*cos(315)},{2*sin(315)}) -- ({2*cos(285)},{2*sin(285)}) -- ({2*cos(270)},{2*sin(270)});

\draw [xshift = 57, yshift = 248] ({1*cos(25)},{1*sin(14)}) -- ({1*cos(100)},{1*sin(14)});
\draw [xshift = 57, yshift = 218] ({1*cos(25)},{1*sin(14)}) -- ({1*cos(100)},{1*sin(14)});

\draw [xshift = -3, yshift = 240] ({2*cos(15)},{2*sin(15)}) -- ({2*cos(45)},{2*sin(45)}) -- ({2*cos(75)},{2*sin(75)}) --  ({2*cos(105)},{2*sin(105)}) -- ({2*cos(135)},{2*sin(135)}) -- ({2*cos(165)},{2*sin(165)}) -- ({2*cos(195)},{2*sin(195)}) -- ({2*cos(225)},{2*sin(225)}) -- ({2*cos(255)},{2*sin(255)}) -- ({2*cos(285)},{2*sin(285)}) -- ({2*cos(315)},{2*sin(315)}) -- ({2*cos(345)},{2*sin(345)}) -- ({2*cos(15)},{2*sin(15)});

\draw [xshift = 197, yshift = 248] ({1*cos(25)},{1*sin(14)}) -- ({1*cos(100)},{1*sin(14)});
\draw [xshift = 197, yshift = 218] ({1*cos(25)},{1*sin(14)}) -- ({1*cos(100)},{1*sin(14)});

%%%%%
\draw [xshift = 317, yshift = 248] ({1*cos(25)},{1*sin(14)}) -- ({1*cos(60)},{1*sin(14)});
\draw [xshift = 317, yshift = 218] ({1*cos(25)},{1*sin(14)}) -- ({1*cos(60)},{1*sin(14)});

\draw [xshift = 320, yshift = 7] ({1*cos(25)},{1*sin(14)}) -- ({1*cos(60)},{1*sin(14)});
\draw [xshift = 320, yshift = -21] ({1*cos(25)},{1*sin(14)}) -- ({1*cos(60)},{1*sin(14)});

\draw [xshift = -153, yshift = 128] ({1*cos(25)},{1*sin(14)}) -- ({1*cos(60)},{1*sin(14)});
\draw [xshift = -153, yshift = 98] ({1*cos(25)},{1*sin(14)}) -- ({1*cos(60)},{1*sin(14)});
%%%%%%

\draw [xshift = 137, yshift = 240] ({2*cos(15)},{2*sin(15)}) -- ({2*cos(45)},{2*sin(45)}) -- ({2*cos(75)},{2*sin(75)}) --  ({2*cos(105)},{2*sin(105)}) -- ({2*cos(135)},{2*sin(135)}) -- ({2*cos(165)},{2*sin(165)}) -- ({2*cos(195)},{2*sin(195)}) -- ({2*cos(225)},{2*sin(225)}) -- ({2*cos(255)},{2*sin(255)}) -- ({2*cos(285)},{2*sin(285)}) -- ({2*cos(315)},{2*sin(315)}) -- ({2*cos(345)},{2*sin(345)}) -- ({2*cos(15)},{2*sin(15)});

\draw [xshift = 277, yshift = 240] ({2*cos(15)},{2*sin(15)}) -- ({2*cos(45)},{2*sin(45)}) -- ({2*cos(75)},{2*sin(75)}) --  ({2*cos(105)},{2*sin(105)}) -- ({2*cos(135)},{2*sin(135)}) -- ({2*cos(165)},{2*sin(165)}) -- ({2*cos(195)},{2*sin(195)}) -- ({2*cos(225)},{2*sin(225)}) -- ({2*cos(255)},{2*sin(255)}) -- ({2*cos(285)},{2*sin(285)}) -- ({2*cos(315)},{2*sin(315)}) -- ({2*cos(345)},{2*sin(345)}) -- ({2*cos(15)},{2*sin(15)});

%%%%%

\draw [xshift = 275, yshift = -142] (-1.45,3.2) -- (-1.25,3.58);
\draw [xshift = 262, yshift = -110] (-.11,1.55) -- (0.1,1.95);

\draw [xshift = 135, yshift = -142] (-1.45,3.2) -- (-1.25,3.58);
\draw [xshift = 122, yshift = -110] (-.11,1.55) -- (0.1,1.95);

\draw [xshift = -5, yshift = -142] (-1.45,3.2) -- (-1.25,3.58);
\draw [xshift = -18, yshift = -110] (-.11,1.55) -- (0.1,1.95);

%%%%

\draw [xshift = -80, yshift = 218] (-1.7,2.7) -- (-1.47,3.1);
\draw [xshift = -93, yshift = 250] (-.35,1.05) -- (-.1,1.5);

\draw [xshift = 60, yshift = 218] (-1.7,2.7) -- (-1.47,3.1);
\draw [xshift = 47, yshift = 250] (-.35,1.05) -- (-.1,1.5);

\draw [xshift = 200, yshift = 218] (-1.7,2.7) -- (-1.47,3.1);
\draw [xshift = 187, yshift = 250] (-.35,1.05) -- (-.1,1.5);

\draw [xshift = 340, yshift = 218] (-1.7,2.7) -- (-1.47,3.1);
\draw [xshift = 327, yshift = 250] (-.35,1.05) -- (-.1,1.5);

%%%%

\draw [xshift = -1, yshift = 290] (-.6,.15) -- (-0.83,.55);
\draw [xshift = -26, yshift = 275] (-.6,.15) -- (-0.83,.55);

\draw [xshift = 139, yshift = 290] (-.6,.15) -- (-0.83,.55);
\draw [xshift = 114, yshift = 275] (-.6,.15) -- (-0.83,.55);

\draw [xshift = 279, yshift = 290] (-.6,.15) -- (-0.83,.55);
\draw [xshift = 254, yshift = 275] (-.6,.15) -- (-0.83,.55);

%\node[ver] () at (-1.2-5,-.5){\tiny $u_{-1,0}$};
%\node[ver] () at (-.8-5,-1.1){\tiny $v_{-1,0}$};
%\node[ver] () at (-.6-5,-1.6){\tiny $w_{-1,0}$};
\node[ver] () at (1-5,-.5){\tiny $x_{-1,-2}$};
\node[ver] () at (-4.3,-1.3){\tiny $y_{-1,-2}$};
\node[ver] () at (-3.3,-2.1){\tiny $z_{-1,-2}$};
\node[ver] () at (1-5,.5){\tiny $u_{-1,-1}$};
\node[ver] () at (-4.3,1.3){\tiny $v_{-1,-1}$};
\node[ver] () at (.1-5,1.7){\tiny $w_{-1,-1}$};
%\node[ver] () at (-1.2-5,.5){\tiny $x_{-1,1}$};
%\node[ver] () at (-.8-5,1.1){\tiny $y_{-1,1}$};
%\node[ver] () at (-.6-5,1.6){\tiny $z_{-1,1}$};

\node[ver] () at (-1.1,-.5){\tiny $u_{0,-2}$};
\node[ver] () at (-.8,-1.3){\tiny $v_{0,-2}$};
\node[ver] () at (-1.5,-2){\tiny $w_{0,-2}$};
\node[ver] () at (1.1,-.5){\tiny $x_{0,-2}$};
\node[ver] () at (.8,-1.3){\tiny $y_{0,-2}$};
\node[ver] () at (1.4,-2.1){\tiny $z_{0,-2}$};
\node[ver] () at (1.1,.5){\tiny $u_{0,-1}$};
\node[ver] () at (.8,1.1){\tiny $v_{0,-1}$};
\node[ver] () at (1.4,1.9){\tiny $w_{0,-1}$};
\node[ver] () at (-1.1,.5){\tiny $x_{0,-1}$};
\node[ver] () at (-.8,1.1){\tiny $y_{0,-1}$};
\node[ver] () at (.1,1.5){\tiny $z_{0,-1}$};

\node[ver] () at (-1.1+5,-.5){\tiny $u_{1,-2}$};
\node[ver] () at (-.8+5,-1.3){\tiny $v_{1,-2}$};
\node[ver] () at (3.5,-2){\tiny $w_{1,-2}$};
\node[ver] () at (1.1+5,-.5){\tiny $x_{1,-2}$};
\node[ver] () at (.7+5,-1.3){\tiny $y_{1,-2}$};
\node[ver] () at (6.4,-2.1){\tiny $z_{1,-2}$};
\node[ver] () at (1.1+5,.5){\tiny $u_{1,-1}$};
\node[ver] () at (.8+5,1.1){\tiny $v_{1,-1}$};
\node[ver] () at (6.3,1.9){\tiny $w_{1,-1}$};
\node[ver] () at (-1.1+5,.5){\tiny $x_{1,-1}$};
\node[ver] () at (-.8+5,1.1){\tiny $y_{1,-1}$};
\node[ver] () at (5.1,1.5){\tiny $z_{1,-1}$};

\node[ver] () at (-1.2+10,-.5){\tiny $u_{2,-2}$};
\node[ver] () at (-.9+10,-1.3){\tiny $v_{2,-2}$};
\node[ver] () at (8.4,-2){\tiny $w_{2,-2}$};
\node[ver] () at (1+10,-.5){\tiny $x_{2,-2}$};
\node[ver] () at (.6+10,-1.3){\tiny $y_{2,-2}$};
\node[ver] () at (11.4,-2.1){\tiny $z_{2,-2}$};
\node[ver] () at (1+10,.5){\tiny $u_{2,-1}$};
\node[ver] () at (.7+10,1.1){\tiny $v_{2,-1}$};
\node[ver] () at (11.3,1.9){\tiny $w_{2,-1}$};
\node[ver] () at (-1.2+10,.5){\tiny $x_{2,-1}$};
\node[ver] () at (-.8+10,1.1){\tiny $y_{2,-1}$};
\node[ver] () at (10,1.5){\tiny $z_{2,-1}$};

%%%
\node[ver] () at (-1.1-2.5,-.5+4){\tiny $u_{-1,0}$};
\node[ver] () at (-1-2.3,-1.1+4){\tiny $v_{-1,0}$};
\node[ver] () at (-4,2.3){\tiny $w_{-1,0}$};
\node[ver] () at (1.2-2.6,-.5+4){\tiny $x_{-1,0}$};
\node[ver] () at (1-2.8,-1.1+4){\tiny $y_{-1,0}$};
\node[ver] () at (-1.1,2.2){\tiny $z_{-1,0}$};
\node[ver] () at (1.2-2.6,.5+4){\tiny $u_{-1,1}$};
\node[ver] () at (1-2.5,1.1+4){\tiny $v_{-1,1}$};
\node[ver] () at (.6-2.5,1.8+4){\tiny $w_{-1,1}$};
\node[ver] () at (-1.2-2.5,.5+4){\tiny $x_{-1,1}$};
\node[ver] () at (-1-2.5,1.1+4){\tiny $y_{-1,1}$};
\node[ver] () at (-.6-2.4,1.6+4){\tiny $z_{-1,1}$};

\node[ver] () at (-1.4+2.5,-.5+4.2){\tiny $u_{0,0}$};
\node[ver] () at (-1.2+2.6,-1.1+4){\tiny $v_{0,0}$};
\node[ver] () at (-.6+3,-1.6+4){\tiny $w_{0,0}$};
\node[ver] () at (1.2+2.5,-.5+4.2){\tiny $x_{0,0}$};
\node[ver] () at (1.1+2.2,-1.1+4){\tiny $y_{0,0}$};
\node[ver] () at (3.5,2.1){\tiny $z_{0,0}$};
\node[ver] () at (1.4+2.3,.5+4.2){\tiny $u_{0,1}$};
\node[ver] () at (1.4+3,1.1+4.4){\tiny $v_{0,1}$};
\node[ver] () at (.5+2.2,1.6+4.2){\tiny $w_{0,1}$};
\node[ver] () at (-1.4+2.5,.5+4.2){\tiny $x_{0,1}$};
\node[ver] () at (-1.2+1.5,1.1+4.4){\tiny $y_{0,1}$};
\node[ver] () at (2.4,6.3){\tiny $z_{0,1}$};

\node[ver] () at (-1.4+7.5,-.5+4.2){\tiny $u_{1,0}$};
\node[ver] () at (-1.2+7.6,-1.1+4){\tiny $v_{1,0}$};
\node[ver] () at (7.2,-1.6+4.1){\tiny $w_{1,0}$};
\node[ver] () at (1.2+7.5,-.5+4.2){\tiny $x_{1,0}$};
\node[ver] () at (1.1+7.1,-1.1+4){\tiny $y_{1,0}$};
\node[ver] () at (8.5,2.1){\tiny $z_{1,0}$};
\node[ver] () at (1.4+7.3,.5+4.2){\tiny $u_{1,1}$};
\node[ver] () at (1.1+7.2,1.1+4.3){\tiny $v_{1,1}$};
\node[ver] () at (8.5,1.6+4.5){\tiny $w_{1,1}$};
\node[ver] () at (-1.4+7.5,.5+4.2){\tiny $x_{1,1}$};
\node[ver] () at (-1.2+7.5,1.1+4.2){\tiny $y_{1,1}$};
\node[ver] () at (7.3,1.6+4.2){\tiny $z_{1,1}$};

\node[ver] () at (-1.6+12.5,-.5+4.2){\tiny $u_{2,0}$};
\node[ver] () at (-1.2+12.5,-1.1+4){\tiny $v_{2,0}$};
\node[ver] () at (-.6+12.5,-1.8+4.2){\tiny $w_{2,0}$};
\node[ver] () at (-1.6+12.5,.5+4.2){\tiny $x_{2,1}$};
\node[ver] () at (-1.2+12.5,1.1+4.4){\tiny $y_{2,1}$};
\node[ver] () at (-.4+12.5,1.6+4.2){\tiny $z_{2,1}$};

%%%%%%

\node[ver] () at (2-6,-.5+8.3){\tiny $x_{-2,2}$};
\node[ver] () at (1.7-6,-1.3+8.5){\tiny $y_{-2,2}$};
\node[ver] () at (.1-5,-1.8+8.5){\tiny $z_{-2,2}$};
\node[ver] () at (1.2-5,.5+8.3){\tiny $u_{-2,3}$};
\node[ver] () at (.8-5,1.3+8.3){\tiny $v_{-2,3}$};
\node[ver] () at (.1-5,1.7+8.3){\tiny $w_{-2,3}$};

\node[ver] () at (-1.2,-.5+8.3){\tiny $u_{-1,2}$};
\node[ver] () at (-.8,-1.1+8.3){\tiny $v_{-1,2}$};
\node[ver] () at (-1.6,-1.6+8){\tiny $w_{-1,2}$};
\node[ver] () at (1.1,-.5+8.5){\tiny $x_{-1,2}$};
\node[ver] () at (2.2,-1.1+8.1){\tiny $y_{-1,2}$};
\node[ver] () at (0,-1.6+8.4){\tiny $z_{-1,2}$};
\node[ver] () at (1.1,.5+8.4){\tiny $u_{-1,3}$};
\node[ver] () at (1.3,10.4){\tiny $v_{-1,3}$};
\node[ver] () at (.6,1.5+8.2){\tiny $w_{-1,3}$};
\node[ver] () at (-1.2,.5+8.3){\tiny $x_{-1,3}$};
\node[ver] () at (-.8,1.1+8.3){\tiny $y_{-1,3}$};
\node[ver] () at (0,1.6+8.5){\tiny $z_{-1,3}$};

\node[ver] () at (-1.5+5,-.5+8.3){\tiny $u_{0,2}$};
\node[ver] () at (-1.3+5,-1.1+8.3){\tiny $v_{0,2}$};
\node[ver] () at (4.8,6.2){\tiny $w_{0,2}$};
\node[ver] () at (1.2+5,-.5+8.3){\tiny $x_{0,2}$};
\node[ver] () at (.7+5,-1.1+8.3){\tiny $y_{0,2}$};
\node[ver] () at (6,-1.6+8){\tiny $z_{0,2}$};
\node[ver] () at (1.2+5,.5+8.3){\tiny $u_{0,3}$};
\node[ver] () at (.7+5,1.1+8.3){\tiny $v_{0,3}$};
\node[ver] () at (6,10.4){\tiny $w_{0,3}$};
\node[ver] () at (-1.6+5.1,.5+8.3){\tiny $x_{0,3}$};
\node[ver] () at (-1.2+5,1.1+8.3){\tiny $y_{0,3}$};
\node[ver] () at (4.6,1.6+8.4){\tiny $z_{0,3}$};

\node[ver] () at (-1.5+10,-.5+8.3){\tiny $u_{1,2}$};
\node[ver] () at (-1.3+10.2,-1.1+8.1){\tiny $v_{1,2}$};
\node[ver] () at (9.8,6.2){\tiny $w_{1,2}$};
\node[ver] () at (1+10,-.5+8.3){\tiny $x_{1,2}$};
\node[ver] () at (.6+10,-1.1+8.2){\tiny $y_{1,2}$};
\node[ver] () at (11,-1.6+8){\tiny $z_{1,2}$};
\node[ver] () at (1.1+10,.5+8.3){\tiny $u_{1,3}$};
\node[ver] () at (.7+10,1.1+8.5){\tiny $v_{1,3}$};
\node[ver] () at (11,10.4){\tiny $w_{1,3}$};
\node[ver] () at (-1.6+10,.5+8.3){\tiny $x_{1,3}$};
\node[ver] () at (-1.2+10,1.1+8.3){\tiny $y_{1,3}$};
\node[ver] () at (-.6+10,1.6+8.4){\tiny $z_{1,3}$};

\draw [thick, dotted] (2.5,4.2) -- (7.5,4.2);
\draw [thick, dotted] (2.5,4.2) -- (5,8.4);
\draw [thick, dotted] (2.5,4.2) -- (0,8.4);

\node[ver] () at (2.5,4.2){$\bullet$};
\node[ver] () at (7.5,4.2){$\bullet$};
\node[ver] () at (5,8.4){$\bullet$};
\node[ver] () at (0,8.4){$\bullet$};

\node[ver] () at (3, -4){\normalsize (b): Truncated trihexagonal tiling $E_{7}$};

\end{tikzpicture}

\caption*{Figure 4}
\end{figure} 

%%%%%%%%%%%%%%%%%%%%%%E6&E7%%%%%%%%%%%%%%%%%%%%%%%%%

%\medskip

\noindent {\bf Claim 5.}  ${\mathcal L}_7 \unlhd {\mathcal G}_7$.

\smallskip

By same arguments as in the proof of Claim 4, $\rho_7\circ\alpha_7^m\circ \rho_7^{-1} = \beta_7^m$,  $\rho_7 \circ \beta_7^m\circ \rho_7^{-1} = \alpha_7^{-m} \in {\mathcal L}_7$. Since $\alpha_7$ and $\tau_7$ commute, $\tau_7 \circ \alpha_7^m\circ \tau_7^{-1} = \alpha_7^m\in {\mathcal L}_7$. 
For $x, y\in \mathbb{R}$,  $(\tau_7\circ\beta_7^m\circ \tau_7^{-1})(xA_7+ yB_7) = ( \tau_7\circ\beta_7^m)(xA_7+y(-F_7)) = ( \tau_7\circ\beta_7^m)(xA_7 +y(A_7-B_7)) = ( \tau_7\circ\beta_7^m)((x+y)A_7 - yB_7) =   \tau_7((x+y)A_7 +(m- y)B_7) =  (x+y)A_7 + (m-y)(-F_7) = (x+y)A_7 + (m-y)(A_7- B_7) = (xA_7 + yB_7) + mA_7 -mB_7 = (\alpha_7^m \circ \beta_7^{-m})(xA_7 + yB_7)$. Thus, $\tau_7\circ\beta_7^m \circ \tau_7^{-1} = \alpha_7^m \circ \beta_7^{-m}\in {\mathcal L}_7$. 
 Since $\alpha_1$ and $\alpha_2$ commute, these prove  Claim 5.

\smallskip

By Claim 5, ${\mathcal G}_7/{\mathcal L}_7$ is a group. Since ${\mathcal L}_7 \unlhd {\mathcal K}_7$, $X = E_7/{\mathcal K}_7$ is a toroidal map and ${\mathcal L}_7$ is generated by two independent vectors, it follows that $Y := E_7/{\mathcal L}_7$ is a toroidal map and $v+ {\mathcal L}_7 \mapsto v + {\mathcal K}_7$ gives a (natural) covering $\eta: Y \to X$. Since the action of ${\mathcal G}_7$ on $E_7$ is vertex-transitive,  the action of ${\mathcal G}_7/{\mathcal L}_7$ on $Y = E_7/{\mathcal L}_7$ is also vertex-transitive. This proves the result when the vertex-type of $X$ is $[4^1,6^1,12^1]$. This completes the proof.  
\end{proof}

%\newpage

\begin{remark} \label{remark1}
{\rm  Let $X$ be a semi-equivelar toroidal map of  vertex-type  $[4^1, 8^2]$ and let  $Y = E_1/{\mathcal L}_1$ be as in the proof of Theorem \ref{theo:se-vt}.  Then  $\eta : Y = E_1/{\mathcal L}_1\to X = E_1/{\mathcal K}_1$ is an $n$-fold cover, where $n = [{\mathcal K}_1 : {\mathcal L}_1] = [{\mathcal H}_1 : {\mathcal L}_1]/[{\mathcal H}_1 : {\mathcal K}_1] = m^2/|ad - bc|$. Since $m$ is a factor of $|ad-bc|$, it follows that $n$ is a factor of $m$ and hence a factor of $|ad-bc|$. If we take $A_1=(1,0)$ and $B_1=(0,1)$ then the area of the parallelogram whose two adjacent sides are $aA_1+bB_1$ and $cA_1 + dB_1$ is $|ad-bc|$. Thus, $|ad-bc|$ is the area of the surface (torus) $|E_1/{\mathcal K}_1|$. 
Similar facts are true for other semi-equivelar toroidal maps also.  For $3\leq i\leq 7$, if $A_i$ and $B_i$ are unit vectors then the area of the torus $|E_i/{\mathcal K}_i|$ is $|ad-bc|\times$(the area of the parallelogram whose two adjacent sides are $A_i$ and $B_i$) $=\frac{\sqrt3}{2}|ad-bc|$. }
\end{remark}

\begin{remark} \label{remark2}
{\rm For any positive integer $r$, let $r{\mathcal K}_i := \langle\alpha_i^{rm}, \beta_i^{rm}\rangle$. Then, by the similar arguments, $Y_{r} := E_i/(r{\mathcal K}_i)$ is also a vertex-transitive cover of $X$ in each of the seven cases. Therefore, $X$ has countably many vertex-transitive covers. 
}
\end{remark}

%%%%%%%%%%%%%%%%%%%%%%%%%%%%%%%%%%%%%%%%%%%%

%\medskip

%

{\small
%% BIBLIOGRAPHY %%

}

\end{document}